\documentclass[11pt]{article}
\usepackage{amsmath,amsthm,amsfonts,amssymb,mathrsfs,bm}
\usepackage{amsmath,amsthm,amsfonts,amssymb,bm,wasysym}
\usepackage{epsfig}
\usepackage[usenames]{color}
\usepackage{verbatim}
\usepackage{hyperref}
\usepackage{multicol}
\usepackage{comment}
\usepackage{float}
\usepackage{graphicx}
\usepackage{centernot}
\usepackage{tikz}
\usepackage{color}
\usepackage[]{algorithm2e}
\usepackage{enumerate}

\graphicspath{{./Pics}}
\usepackage[color=green]{todonotes}

\usepackage[normalem]{ulem}

\parskip \medskipamount
\newtheorem{theorem}{Theorem}[section]

\numberwithin{equation}{section}
\newtheorem{lemma}[theorem]{Lemma}
\newtheorem{proposition}[theorem]{Proposition}
\newtheorem{corollary}[theorem]{Corollary}

\newtheorem{remark}[theorem]{Remark}

\newtheorem{claim}[theorem]{Claim}

\newtheorem{question}[theorem]{Question}
\newtheorem{conjecture}[theorem]{Conjecture}
\numberwithin{equation}{section}

\def\N{\mathbb{N}}

\def\R{\mathbb{R}}

\renewcommand{\phi}{\varphi}
\renewcommand{\epsilon}{\varepsilon}

\allowdisplaybreaks

\newcommand{ \mix}{ t_{\mathrm{mix}} }

\newcommand{\1}{{\text{\Large $\mathfrak 1$}}}

\newcommand{ \rel}{ t_{\mathrm{rel}} }

\newcommand{\til}{\widetilde}

\newcommand{\tmix}{t_{\mathrm{mix}}}

\newcommand{\tsep}{t_{\mathrm{sep}}}

\newcommand{\trel}{t_{\mathrm{rel}}}

\newcommand{\pr}[1]{\mathbb{P}\!\left(#1\right)}
\newcommand{\E}[1]{\mathbb{E}\!\left[#1\right]}
\newcommand{\estart}[2]{\mathbb{E}_{#2}\!\left[#1\right]}
\newcommand{\prstart}[2]{\mathbb{P}_{#2}\!\left(#1\right)}
\newcommand{\prcond}[3]{\mathbb{P}_{#3}\!\left(#1\;\middle\vert\;#2\right)}
\newcommand{\econd}[2]{\mathbb{E}\!\left[#1\;\middle\vert\;#2\right]}

\def\cR{\mathcal{R}}

\newcommand{\tn}{|\kern-.1em|\kern-0.1em|}

\newcommand\be{\begin{equation}}
\newcommand\ee{\end{equation}}

\def\la{\lambda}
\newcommand{\tv}[1]{\left\|#1\right\|_{\rm{TV}}}

\newcommand{\tcov}{t_{\mathrm{cov}}}

\newcommand{\taucov}{\tau_{\mathrm{cov}}}
\newcommand{\reff}{R_{\rm{eff}}}

\begin{document}

\title{\bf Covering a graph with independent walks}

\author{Jonathan Hermon
\thanks{University of British Columbia, Vancouver, Canada. E-mail:  {jhermon@math.ubc.ca}. Financial support by NSERC grants.}
\and Perla Sousi
\thanks{University of Cambridge, Cambridge, UK. E-mail:  {p.sousi@statslab.cam.ac.uk}.}
}
\date{}
\maketitle

\begin{abstract}

Let $P$ be an irreducible and reversible transition matrix on a finite state space $V$ with invariant distribution $\pi$. We let $k$ chains start by choosing independent locations distributed according to $\pi$ and then they evolve independently according to $P$. Let $\tau_{\mathrm{cov}}(k)$ be the first time that every vertex of $V$ has been visited at least once by at least one chain and let $\tcov(k)=\E{\tau_{\mathrm{cov}}(k)}$ with $\tcov=\tcov(1)$. We prove that $\tcov(k)\lesssim \tcov/k$. When $k\leq \tcov/\trel$, where $\trel$ is the inverse of the spectral gap, we show that this bound is sharp. For $k\leq \tcov/\tmix$ with $\tmix$ the total variation mixing time of $(P+I)/2$ we prove that $k \cdot \max_{x_1,\ldots,x_k}\estart{\taucov(k)}{x_1,\ldots,x_k} \asymp \tcov$.
\newline
MSC 2010 \emph{subject classifications.} Primary        60J10, 60J27.
\end{abstract}

\section{Introduction}

Let $P$ be an irreducible transition matrix on the finite state space $V$ and suppose that $P$ is reversible with respect to the invariant distribution $\pi$. 

Let $X^1,\ldots, X^k$ be $k$ independent discrete time Markov chains with matrix $P$. For every~$i$ and~$x\in V$ let $\tau_x^i$ be the first hitting time of $x$ by $X^i$, i.e.\
\[
\tau_x^i = \min\{t\geq 0: X_t^i=x\}.
\]
We let $\taucov(k)$ be the first time that every state $x$ of $V$ has been visited at least once by one of the chains. More formally,
\[
\taucov(k) = \min\{ t\geq 0: \ \forall \ x, \ \exists\ i\leq k \ \text{ s.t. } \tau_x^i \leq t \} = \max_ x \min_{i\leq k} \tau_x^i.
\]
When $k=1$, we write $\taucov = \taucov(1)$ and $\tcov = \max_x \estart{\taucov}{x}$.

When $k>1$, we write 
\[
\estart{\taucov(k)}{x_1,\ldots,x_k}=\econd{\taucov(k)}{X_0^1=x_1,\ldots,X_0^k=x_k}
\]
and we also denote by $\estart{\taucov(k)}{\pi^{\otimes k}}$ the expectation of $\taucov(k)$ when the walks $X^i$ start independently according to $\pi$.
The problem of bounding $\max_x \estart{\taucov(k)}{x,\ldots,x}$ in terms of $\tcov$ was first systematically studied in~\cite{AlonKozmaetal}, where they obtained bounds on the speed-up defined by 
\[
S^k(P) = \frac{\tcov}{\max_{x_1,\ldots,x_k}\estart{\taucov(k)}{x_1,\ldots,x_k}}
\]
for random walks on several classes of graphs. We now state a conjecture from~\cite{AlonKozmaetal}.

\begin{conjecture}[\cite{AlonKozmaetal}]
There exist universal constants $C$ and $C'$ so that for any graph $G$ and a simple random walk on $G$ with transition matrix $P$ for all $k$
\[
C'\log k \leq S^k(P) \leq C k.
\]
\end{conjecture} 

Some previous results in the direction of the above conjecture were obtained earlier in~\cite{BroderKarlinetal}. There the authors determined the order of the cover time of $k$ independent walks started from stationarity in certain graphs.

 In the present paper, our main result establishes a bound on the expectation of $\taucov(k)$ when the $k$ chains start independently according to $\pi$. Namely that for all $k$
\[\widetilde{S}^k=\frac{\tcov}{\estart{\taucov(k)}{\pi^{\otimes k}}} \ge c k, \]
for some absolute constant $c>0$. We also show that if $k \le \tcov/\rel$, where $\rel$ is the relaxation time of $P$ defined below in~\eqref{eq:deftrel}, then $\widetilde{S}^k \le Ck$, for some absolute constant $C$ independent of $k$ and $P$. Finally, we also show that $S^k(P)\asymp k$ whenever  $k \le \tcov/\tmix$, where $\tmix$ is the total variation mixing time of the lazy version of $P$ defined below.

\begin{theorem}\label{thm:indepone}
        There exists a positive constant $C$ so that if $P$ is an irreducible and reversible transition matrix on a finite state space with invariant distribution $\pi$, then for all $k$
                        \[
        \estart{\taucov(k)}{\pi^{\otimes k}}
        \leq  C\cdot \left\lceil\frac{\tcov}{k}\right\rceil.
        \]
\end{theorem}

As we mentioned above, the problem of bounding the cover time of $k$ walk started from stationarity was also considered in~\cite{BroderKarlinetal} for special cases of graphs.  Moreover, in~\cite[Chapter 6, Proposition~6.17]{AF} building on techniques of~\cite{BroderKarlinetal} they obtain an upper bound on the cover time of $k$-walks from stationarity for regular graphs.

\begin{remark}
        \rm{
     A very minor modification to our proof of Theorem~\ref{thm:indepone} in fact implies a stronger result. In the setup of Theorem~\ref{thm:indepone}, for every $\delta \in (0,1)$ there exists a positive constant $C=C(\delta)$ (independent of the chain) such that for all $k$ with probability at least $1-\delta$ after $ C\lceil\tcov/k \rceil$ steps of the $k$ chains each state $v$ is visited between $ C\delta\tcov \pi(v)$ and $  C\lceil\delta^{-1}\tcov \pi(v) \rceil$ number of times (by the $k$ chains combined). }
\end{remark}

The problem of bounding $\estart{\taucov(k)}{\pi^{\otimes k}}$ was also considered by Efremenko and Reingold in~\cite{EfremenkoReingold} where the following bound was obtained for any random walk on a finite connected graph $G$ (Theorem~4.8)
\[
\tcov \leq k \estart{\taucov(k)}{\pi^{\otimes k}} + O(\tmix \cdot k\log k ) + O(k\sqrt{\tmix \estart{\taucov(k)}{\pi^{\otimes k}}}).
\]

Related results were also obtained by Els\"{a}sser and Sauerwald \cite{Sauer}. The value of $\estart{\taucov(k)}{\pi^{\otimes k}}$ was determined up to smaller order terms
         for a wide range of $k$ in \cite{frogs1}  in the case where $G$ is a $d$-dimensional (discrete grid) torus of side length $n$ (see also \cite[Proposition 2.8]{frogs1} for a certain general result about $\taucov(k)$). It is conjectured in \cite{frogs1} for vertex transitive graphs and proved in the case of tori that the cover time by multiple walks starting from i.i.d.\ stationary initial positions is intimately related to the susceptibility of the frog model on the same graph. Loosely speaking, the susceptibility of the frog model is the minimal lifetime of an infected individual which is sufficient to ensure that the infection reaches all particles before dying out, if initially only the particles at the origin are infected. The cover time by multiple random walks also has some algorithmic applications. We refer the reader to \cite{Thomas1} for a discussion of such applications. For a comprehensive discussion of the existing literature about cover times see \cite{DingLeePeres}.

Before stating the next result, we recall the definition of the total variation mixing time. For two probability measures $\mu$ and $\nu$ we write $\tv{\mu-\nu}$ for the total variation distance between $\mu$ and~$\nu$. For a transition matrix $P$ we write  $P_L=(P+I)/2$ for the lazy version of $P$ in order to avoid periodicity and near-periodicity issues. For every $\epsilon\in (0,1)$, the $\epsilon$-total variation mixing time is defined to be 
\[
\tmix(\epsilon)=\min\{t\geq 0: \max_x\tv{P_L^t(x,\cdot) - \pi} \leq \epsilon\}.
\]
We write $\tmix = \tmix(1/4)$.

\begin{corollary}\label{cor:worststartingstate}
There exists a positive constant $C$ so that if $P$ is an irreducible and reversible transition matrix on a finite state space $V$, then for all $k$ we have 
\begin{align*}
        \max_{x_1,\ldots, x_k} \estart{\tau_{\rm{cov}}(k)}{x_1,\ldots,x_k} \leq C\left( \tmix + \frac{\tcov}{k}\right).
\end{align*}
\end{corollary}

In the following proposition we prove that when $k\leq \tcov/\tmix$, then the speed-up is linear. The upper bound follows from Corollary~\ref{cor:worststartingstate}, while the proof of the lower bound is similar to the proof of \cite[Lemma 4.9]{EfremenkoReingold}.

\begin{proposition}\label{lem:loweboundontcovk}
        There exist constants $C_1,C_2>0$  so that if $P$ is a finite irreducible and reversible transition matrix with invariant distribution $\pi$ and $k\leq \tcov/(16\tmix)$, then 
        \[
C_1\frac{\tcov}{k}\leq  \estart{\taucov(k)}{\pi^{\otimes k}}\leq  \max_{x_1,\ldots,x_k}\estart{\taucov(k)}{x_1,\ldots,x_k} 
\leq C_2 \frac{\tcov}{k}.
        \]
\end{proposition}

In the context of random $d$-regular graphs, a stronger asymptotic form of the second statement in the above proposition   has been derived in~\cite{cooperfriezeetal}.

Before stating our final result which strengthens the first statement of Proposition~\ref{lem:loweboundontcovk}, we recall that the relaxation time $\trel$ of a Markov chain with transition matrix $Q$ is defined to be 
\begin{align}\label{eq:deftrel}
\trel = \frac{1}{\gamma},
\end{align}
where $\gamma$ is the spectral gap given by
\[
\gamma = 1-\max\{ \lambda: \lambda \text{ is an eigenvalue of $Q$ with } \lambda \neq 1\}.
\]

\begin{theorem}
\label{thm: cover time lower bound}
There exist two positive constants $c$ and $C$ so that for all finite irreducible reversible Markov chains $P$ with invariant distribution $\pi$ and for all $k \le c\cdot  t_{\mathrm{cov}}/\trel $ then
\[ 
\frac{1}{C}\cdot \frac{t_{\mathrm{cov}}}{k} \le \estart{\taucov(k)}{\pi^{\otimes k}}
\le C\cdot \frac{t_{\mathrm{cov}}}{k}.  \]

\end{theorem}
The example from Figure 1 in \cite{EfremenkoReingold}, namely two cliques of size $n$ connected by a single edge, shows that in general $\estart{\taucov(k)}{\pi^{\otimes k}}
\ge c \frac{t_{\mathrm{cov}}}{k}$ may fail for $k >  t_{\mathrm{cov}}/\rel$. Indeed, in this example  $\tcov=\Theta(n^2)=\rel$ and one can show that  $\estart{\taucov(k)}{\pi^{\otimes k}}
\leq  c \frac{n \log n }{k}+n^2 e^{-ck}$.

It would be interesting to establish analogous results for non-reversible chains. 
\begin{question}\rm{
	Do the assertions of Theorems~\ref{thm:indepone} and~\ref{thm: cover time lower bound} hold without the assumption of reversibility, where $\trel=1/\gamma$ and $\gamma$ is defined as in Chatterjee~\cite{Chatterjee}?}
\end{question}
\begin{question}\rm{
		Does the assertion of Proposition~\ref{lem:loweboundontcovk} hold without the assumption of reversibility, with $\tmix$ replaced with $\min\{\tmix(1/k),\tsep \}$, where $\tsep:=\min\{t \ge 0:\min_{x,y}P_L^t(x,y)/\pi(y) \ge 3/4 \}$?}
\end{question}

We note that in a recent work by Rivera et al.\ \cite{Thomas}, it is shown that there exists a $c>1/26$ such that  $\estart{\taucov(k)}{\pi^{\otimes k}}\geq
e^{-10}c^k \estart{\taucov}{\pi}$ for any $k$. See \cite{Thomas} for additional results concerning $\estart{\taucov(k)}{\pi^{\otimes k}}$ and $\max_{x_1,\ldots,x_k}\estart{\taucov(k)}{x_1,\ldots,x_k} $. We note that one can generalise Corollary \ref{cor:worststartingstate} by replacing the term $\mix$ on the r.h.s.\ of the display from the corollary with the notion of ``mixing time of $\ell<k$ walks out of $k$" introduced in \cite{Thomas} (which is smaller than $\mix$), and replacing the term $\frac{\tcov}{k}$ by $\frac{\tcov}{\ell}$. One can then take a minimum over $1 \le \ell \le k$.

Theorem \ref{thm: cover time lower bound} shows that Theorem \ref{thm:indepone} is sharp up to a constant factor for $k \le c \tcov/\trel$. We now present two conjectures aimed at describing the regime $k > c \tcov/\trel$.
	\begin{conjecture}
		\label{conj: monotonicity in k} There exists an absolute constant $C$ (independent of $P$) such that for all $k'<k$ we have that  
		\begin{align}\label{conj}
			\estart{\taucov(k)}{\pi^{\otimes k}}    \leq  C\cdot \left\lceil\frac{k'}{k}\estart{\taucov(k')}{\pi^{\otimes k'}}
			\right\rceil.
			\end{align}
	\end{conjecture}
	The conjecture can be restated as saying that the expected combined lengths of the $k$ stationary walks until the cover time, denoted by $f(k)$, satisfies that $f(k) \le Cf(k')$ for all $k' \le k$.
	 The following conjecture is analogous to the previous one, but instead of fixing the number of walks to be $k$ and the random quantity considered to be their length, it concerns the case that the length of the walks $M+1$ is fixed, and the random quantity is the number of  stationary walks of length $M+1$ required to cover the state space. 

	For $p \in [0,1]$ let $K_p=p \Pi + (1-p)P$, where $\Pi$ is the matrix whose rows are all equal to $\pi$. Then for all $x,y$ we have 
	\begin{align}\label{eq:defkp}
	K_p(x,y) = p\pi(y) +(1-p)P(x,y).
	\end{align}
	Denote the worst-case expected cover time for $K_p$ by $\tcov(K_p)$. Let $(X^{(i)}:i \ge 1)$ be a collection of independent realizations of the Markov chain with transition matrix $P$, started from the stationary distribution. For $M \in \mathbb{Z}_+$ let $\tau_M:=\inf\{k: V=\cup_{i=1}^k \{X_j^{(i)}:0 \le j \le M \} \}$ be the number of independent stationary random walks of length $M+1$ (counting time 0 as part of the length) required to cover~$V$. 
		\begin{conjecture}
		\label{conj: monotonicity in length} There exist absolute constants $c,C$ (independent of $P$) such that for all $0 \le M'<M$ and all $k \ge 1$ we have that  
		\begin{align}\label{conj2}
		(M'+1)\mathbb{E}[\tau_{M'}]   \leq  C	(M+1)\mathbb{E}[\tau_{M}]
		.
		\end{align}
				\begin{align}\label{conj3}
		c \tcov(K_{1/(M+1)}) \le	(M+1)\mathbb{E}[\tau_{M}]   \leq  C \tcov(K_{1/(M+1)})	.
		\end{align}
						\begin{align}\label{conj4}
		\mathbb{E}[\tau_{\lceil	\estart{\taucov(k)}{\pi^{\otimes k}}  \rceil}]   \leq  C k.
		\end{align}
	\end{conjecture}

The first inequality in the conjecture has a similar interpretation as \eqref{conj}. The second display suggests that the expectation of  $\tau_M$ would change by at most a constant factor if we modified the definition of $\tau_M$ such that instead of each walk performing a walk of length $M+1$, each walk would instead be of a random length, distributed as the Geometric distribution with parameter $1/(M+1)$. One can interpret \eqref{conj4} as asserting that (up to a constant factor) fixing the lengths of the stationary walks gives a larger speed up than fixing their number.

We believe that the methods developed to prove Theorem \ref{thm:indepone} can be useful in order to prove that there exist absolute constants $c_1,c_2,C_1,C_2>0$ (independent of $P$) such that for all $M \ge 0$
\begin{align}\label{conj3'}
	c_1 \tcov(K_{\frac{1}{c_2(M+1)}}) \le	(M+1)\mathbb{E}[\tau_{M}]   \leq  C_1 \tcov(K_{\frac{1}{C_2(M+1)}}).
\end{align}
Combined with the following bounds that we sketch afterwards, the above would establish~\eqref{conj3}.  One would then be able to derive \eqref{conj2} by combining \eqref{conj3} and \eqref{stability}.  There exists an absolute constant $C'>0$ such that for all $0 \le M' \le M$,
\begin{align}\label{stability}
	\frac{M'+1}{M+1} \tcov(K_{1/(M+1)}) \le	\tcov(K_{1/(M'+1)}) \le C' \tcov(K_{1/(M+1)}).
\end{align}

The second inequality in \eqref{stability} can be derived from a comparison of effective resistances, similar to the one in the proof of Theorem \ref{thm:indepone} (more precisely, the cases $M' \ge 1$ and $M'<1$ need to be treated separately; the case $M' \ge 1$ is very similar to the analysis from the proof of Theorem \ref{thm:indepone}), while the case $M'<1$ requires a different argument).  It is not hard to show that $p \mapsto p\tcov(K_{p})$ is continuous and non-decreasing in $p$, as by Wald's equation this is the expected number of independent stationary random walks of length Geometric$(p)$ required to cover $V$. This implies the first inequality in \eqref{stability}.

It is plausible that the ideas from the proof of Theorem \ref{thm:indepone} can be helpful in order to show that there exist absolute constants $c,c' \in (0,1) ,C,C' \ge  1$ (independent of $P$) such that for all $k$
\begin{equation}
	\label{e:conj5}
	c \tcov(K_{p_1(k)}) \le  k \lceil	\estart{\taucov(k)}{\pi^{\otimes k}}  \rceil \le C \tcov(K_{p_2(k)}),
\end{equation}
where $p_1(k)$ and $p_2(k)$ are given by the equations $p_1(k)\tcov(K_{p_1(k)})=C'k$  and \\ $p_2(k)\tcov(K_{p_2(k)})=c'k$. Since $p \tcov(K_p)$ is continuous and non-decreasing in $p$ such solutions exist, provided that $C'k \le \tcov(\Pi)$, which holds whenever $C'k \le (1-o(1))|V| \log |V|$ by a classic result of Feige~\cite{Feige} that in the reversible setup the expected cover time is always at least $(1-o(1))|V| \log |V|$. We expect that $\lceil	\estart{\taucov(k)}{\pi^{\otimes k}}  \rceil = O(1)$ whenever $k=\Omega(\tcov(\Pi))$ (the implicit constant in the $O(1)$ depends on the one from the $\Omega(\tcov(\Pi))$).  

The first inequality in \eqref{e:conj5} should probably hold with $p_1(k)$ above replaced with $\left( \lceil C'	\estart{\taucov(k)}{\pi^{\otimes k}}  \rceil \right)^{-1}$. This would imply \eqref{conj4} when combined with  \eqref{conj3'} and \eqref{stability}. Unfortunately, one cannot derive \eqref{conj} by combining \eqref{e:conj5} and \eqref{stability} because it is possible that $p_2(k)$ is of smaller order than $p_1(k)$. For instance, it is not hard to verify that if $P$ is the transition matrix of a simple random walk on the graph obtained by connecting two cliques of size $n$ by a single edge, then $|p\tcov(K_p)-2|=o(1)$ whenever $(1+o(1))n \log n \le 1/p =o(n^2)$. In this case, for $k=2$ the leftmost term in \eqref{e:conj5} is of order at most $n \log n$ whereas the rightmost term is of order at least $n^2$.

\emph{Notation} For functions $f$ and $g$ we write $f(n) \lesssim g(n)$ if there exists a constant $c > 0$ such that $f(n) \leq c g(n)$ for all $n$.  We write $f(n) \gtrsim g(n)$ if $g(n) \lesssim f(n)$.  Finally, we write $f(n) \asymp g(n)$ if both $f(n) \lesssim g(n)$ and $f(n) \gtrsim g(n)$.

\textbf{Acknowledgements:} We thank Nicol\'as Rivera, Thomas Sauerwald and John Sylvester for useful discussions and for feedback on an earlier draft of this work.

\subsection{Overview}
\label{s:overview}

In this section we first give an outline of the argument used to prove Theorem~\ref{thm:indepone} and then a brief discussion on the proofs of Corollary \ref{cor:worststartingstate} and Theorem \ref{thm: cover time lower bound}.

Our proof of Theorem~\ref{thm:indepone} builds upon the groundbreaking work of Ding, Lee and Peres~\cite{DingLeePeres}, where they established that the expected cover time of a random walk on a weighted graph is up to universal constants equal to the total conductance of the graph multiplied by the square of the expected maximum of the discrete Gaussian free field on the graph. Moreover, they prove that the expected cover time is comparable to the strong $\delta$-blanket time, which is the first time $t$ such that the chain visits every state at least $\delta t \pi(v)$ times and at most $\delta^{-1}t\pi(v)$ times by time $t$. This result is crucial for our analysis as we explain below.

In order to prove Theorem~\ref{thm:indepone} we employ the following strategy:
\begin{itemize}
\item[(i)]
Construct an auxiliary Markov chain whose expected cover time is at most of order $k\estart{\taucov(k)}{\pi^{\otimes k}}$ up to some universal constants.
\item[(ii)]
Prove that the expected cover time of the auxiliary chain is at most $C\tcov$ for some absolute positive constant $C$, where $\tcov$ is the expected cover time of the original chain. 
\end{itemize}

We actually construct a whole family of auxiliary Markov chains indexed by a parameter $\lambda\in (0,1)$. We do so by adding a new state $\partial$ to the state space $V$ and modify the transition matrix by allowing jumps to $\partial$ with probability $\lambda$ and with probability $1-\lambda$ the Markov chain uses the original matrix $P$. When at $\partial$ the chain jumps to $x\in V$ with probability $\pi(x)$.

We now explain how we address (ii). Using the equivalence from~\cite{DingLeePeres} between the cover time and the maximum of the GFF, the problem reduces to a comparison between the expected maxima of two Gaussian processes. One obstacle in comparing them though is that the two GFF's are defined on different state spaces due to the addition of $\partial$. However, this turns out to not be a major obstacle. One solution is to consider $\lambda _{0}\asymp 1/\tcov$ and show that for such a choice the cover time of the auxiliary chain is comparable to that of the original chain. One can then compare the effective resistances of the auxiliary chain for smaller values of $\lambda$ to those of the auxiliary chain with parameter~$\lambda_{0}$. Applying then the Sudakov-Fernique inequality translates a comparison between effective resistances into a comparison between the expectations of the maxima of the two GFF's, and hence this yields a comparison of expected cover times.

We now discuss (i). Naturally, when considering  $\estart{\taucov(k)}{\pi^{\otimes k}}$ it is natural to take $\lambda$ such that $k/\la \asymp \tcov$, since the first $k$ excursions from $\partial$ to itself take roughly $k/\lambda$ time units for the chain to complete. To bound   $\estart{\taucov(k)}{\pi^{\otimes k}}$ from above by $C \tcov/k$ it does not suffice to argue  that for such choice of $\lambda$ the union of the first $k$ excursions covers the state space.  The difficulty  is that while the lengths of the excursions have mean $1/\lambda$, these lengths  are random, and some of them will typically be unusually long. Namely, out of the first $k$ excursions, the longest excursion will typically be of length $\asymp(\log k) /\lambda$.

To address this issue we show that even if we truncate excursions of length above $C_{0}/\lambda$ so that their new length is  $C_{0}/\lambda$,   for some large absolute constant $C_0>0$,  the union of the first $k$ (possibly truncated) excursions covers the graph with  probability bounded away from zero. The idea is roughly as follows: If we set $\lambda$ such that  $k/\lambda=C_{1} \tcov$ for a large constant $C_{1}$ then from the aforementioned blanket time result we have that with probability close to $1$ by taking $C_{1}$ large enough in the first $k$ excursions every state $v$ is visited at least $0.99\pi(v)k/\lambda$ times. Here we use the choice   $k/\lambda=C_{1} \tcov$ together with (ii) above and we also need to take $C_1$ to be large compared with $C$ from (ii).

To conclude the proof it would now suffice to show that during the union of the parts of the excursions after time $C_0/\lambda$ (here time is measured w.r.t.\ each excursion)  there is no vertex $v$ which is visited at least  $0.98\pi(v)k/\lambda$ times. Crucially, only a relatively few excursions (roughly $e^{-C_0}k(1 \pm o(1))$ of them) would  be of length at least $C_0/\lambda$. Using the memoryless property of the Geometric distribution, as well as the stationarity of the excursions, we can apply the upper bound on the local time at a state from the definition of the blanket time to the union of the parts of the excursions after time  $C_0/\lambda$. To make this argument rigorous we cannot directly work with the blanket time, since its definition lacks monotonicity. To overcome this issue, in Section~\ref{sec:modifiedchain} we state and prove some results on local times whose proofs follow similarly to~\cite{DingLeePeres}. This concludes the discussion of the proof of Theorem \ref{thm:indepone}.

We now discuss Corollary \ref{cor:worststartingstate}. First we may assume that the original chain is lazy as this can only increase $\max_{x_1,\ldots,x_k}\estart{\taucov(k)}{x_1,\ldots,x_k}$, and this increases $\tcov$ by exactly a factor 2.  The idea is to let the $k$ chains first evolve for $C \mix$ time units, where $\mix$ is as above the mixing time of the lazy chain. Using standard results on reversible chains we argue that after this initial burn in period we may assume that at least $k/2$ of the walks have i.i.d.\ stationary locations at time  $C \mix$. 
The proof can then be concluded using Theorem \ref{thm:indepone}.

We conclude this overview with a brief discussion on the proof of Theorem \ref{thm: cover time lower bound}. By the above discussion, the main ingredient in the proof is to show that if $\lambda_0$ is as above, and $\lambda_0 \le \lambda \le 1/\rel$ then the effective resistances for the auxiliary chain with parameter~$\lambda$ are pairwise comparable up to a universal constant to the corresponding effective resistances for the auxiliary chain with parameter $\lambda_0$. Instead of working with the auxiliary chains we defined above, for technical reasons we actually establish such a comparison of effective resistances between the original chain $P$ and the auxiliary chain $K_{\lambda}$ as defined in~\eqref{eq:defkp}.
We believe that the comparison of effective resistances that we establish is of interest in its own right.

\emph{Organisation} In Section~\ref{sec:prelim} we recall some background on the Gaussian free field (GFF) and the equivalence between the cover time and the maximum of the GFF as established in~\cite{DingLeePeres}. In Section~\ref{sec:modifiedchain} we introduce a family of auxiliary chains and prove bounds on their cover and local times. In Section~\ref{sec:proofformultiplewalks} we give the proofs of Theorem~\ref{thm:indepone}, Corollary~\ref{cor:worststartingstate} and Proposition~\ref{lem:loweboundontcovk}. Finally in Section~\ref{sec:proofoftrel} we prove Theorem~\ref{thm: cover time lower bound} by first establishing a comparison result, Theorem~\ref{thm:effective resistance comparison}, between effective resistances for another auxiliary family of Markov chains that we define there.

\section{Preliminaries}\label{sec:prelim}

We recall some results on the correspondence between the Gaussian free field and cover times obtained by Ding, Lee and Peres in~\cite{DingLeePeres}. 
In this section, $G=(V,E)$ will always be a finite (undirected) connected graph with weights~$w(e)$ assigned to the (undirected) edges. For a vertex $x$ we write $w(x) = \sum_{y\sim x} w(x,y)$. A continuous time weighted random walk $X$ on $G$ is the process that  stays at every vertex for an independent exponential random variable of parameter $1$ and then jumps to a neighbour with probability proportional to the weight of the connecting edge, i.e.
\[
P(x,y) = \frac{w(x,y)}{w(x)}.
\]
The Gaussian free field (GFF) $\eta$ on $G =(V,E)$ with boundary (sometimes also called ground state) $z\in V$ is the zero mean Gaussian process $(\eta_x)_{x\in V}$ indexed by the vertices of $G$ with $\eta_z=0$ and covariances given by 
\begin{equation}
\label{e:covforGFF}
{\rm{Cov}}(\eta_x,\eta_y) = \frac{\estart{L_{\tau_z}(y)}{x}}{w(y)}, \ \forall \ x,y,
\end{equation}
where $L_t(y)=\int_0^t \1(X_s=y)\,ds$ for every $t>0$ is the local time at $y$ and $\tau_z$ is the first hitting time of $z$. Equivalently, for all $x,y \in V$ the variance of $\eta_x-\eta_y$ is given by
\[
\E{(\eta_x - \eta_y)^2} = \reff(x,y),
\]
where $\reff(x,y)$ is the effective resistance between $x$ and $y$ in the weighted graph $G$ (see for instance~\cite[Lemma~2.1]{DingLeePeres}). 

We next state the isomorphism theorem that first appeared in~\cite{isomorphismthm} which will be used several times in the paper and was also used extensively in~\cite{DingLeePeres}.

\begin{theorem}[Generalised Second Ray-Knight Isomorphism theorem, \cite{isomorphismthm}]\label{thm:isomorphism}
Let $G=(V,E)$ be a finite connected graph with weights $(w(e))_{e\in E}$ assigned to the edges and let $z\in V$ be a distinguished vertex of $G$. 
        Let $X$ be a continuous time weighted random walk on $G$ starting from $z$. For every $t>0$ we set 
        \[
        \tau(t) = \inf\left\{s\geq 0: \frac{L_s(z)}{w(z)} \geq t\right\}.
        \]
        Let $\eta$ be a GFF on $G$ with boundary $z$ independent of $X$. Then 
        \[
        \left(\frac{L_{\tau(t)}(x)}{w(x)} + \frac{1}{2} \eta_x^2 \right)_{x\in V} \overset{\mathcal{L}}{=} \left(\frac{1}{2}(\eta_x + \sqrt{2t})^2\right)_{x\in V}.
        \]
\end{theorem}

\begin{theorem}[\cite{DingLeePeres}, Theorem~1.9]\label{thm:dingleeperes}
There exist universal positive constants $c_1$ and $c_2$ so that the following holds.
        Let $G=(V,E)$ be a finite connected graph endowed with weights $(w(e))_{e\in E}$. Let $X$ be a discrete or continuous weighted walk on $G$ and let $\taucov$ be its cover time. Let $\eta$ be a GFF on~$G$ with boundary $z\in V$. Then 
        \[
       c_2\cdot \mathcal{C}(G) \left(\E{\max_x \eta_x}\right)^2 \leq  \max_{x} \estart{\taucov}{x} \leq c_1 \cdot \mathcal{C}(G) \left(\E{\max_x \eta_x}\right)^2,
        \]
        where $\mathcal{C}(G)=\sum_v w(v)$. 
\end{theorem}

\begin{remark}
        \rm{
        We note that Alex Zhai in~\cite{Zhai} sharpened the above result by obtaining exponential concentration bounds for the cover time of general graphs in terms of the expectation of the maximum of the GFF. His bounds are sharp for sequences of graphs with $\max_{x,y}\estart{\tau_y}{x} = o(\tcov)$.
        }
\end{remark}

\begin{lemma}[\cite{DingLeePeres}, Theorem (MM) and Lemma~2.4] \label{lem:dingleeperes}
There exist positive constants $c$ and $u_0$ so that if $G=(V,E)$ is a finite connected graph with weights $(w(e))_{e\in E}$ on the edges and $\eta$ is a GFF on $G$ with boundary~$z\in V$, then for all $u\geq u_0$ we have 
\[
\sum_{x\in V} \exp\left( - u \cdot w(x) \cdot \left(\E{\max_v \eta_v}\right)^2   \right) \lesssim e^{-c u}.
\]
\end{lemma}

Finally we recall the Sudakov-Fernique inequality which gives a comparison for the maxima of two Gaussian processes given a condition on their respective variances.

\begin{proposition}{\rm{(\cite[Proposition~5.11]{Biskupnotes})}}\label{pro:sudakovfernique} 
Suppose that $X$ and $Y$ are centred Gaussian processes in $\R^n$ satisfying 
	\[
	\E{(X_i-X_j)^2}\leq \E{(Y_i-Y_j)^2} \ \forall \ i,j=1,\ldots,n.
	\]
	Then we have 
	\[
	\E{\max_{i\leq n} X_i}\leq \E{\max_{i\leq n}Y_i}.
	\]
\end{proposition}

\section{Cover time of an auxiliary Markov chain}
\label{sec:modifiedchain}

Let $P$ be an irreducible transition matrix on the finite state space $V$ and suppose it is reversible with respect to the invariant distribution $\pi$. In this section we introduce a family of auxiliary Markov chains indexed by $\lambda\in (0,1)$ by adding the state $\partial$ to the state space $V$ and modifying the transition matrix $P$ as follows: for every $\lambda\in (0,1)$ we let 
\[
{P}^\lambda(x,\partial) = \lambda, \quad {P}^\lambda(x,y) = (1-\lambda)P(x,y) \quad \text{ and } \quad {P}^\lambda(\partial, x) = \pi(x) \quad \forall \ x, y\in V.
\] 
Then it is immediate to check that $P^\lambda$ is reversible with respect to $\pi^\lambda$ given by 
\[
{\pi}^\lambda(x) = \frac{\pi(x)}{1+\lambda} \quad \text{for} \ x\neq \partial \quad  \text{ and }\quad  {\pi}^\lambda(\partial) = \frac{\lambda}{1+\lambda}.
\]
Since the matrix ${P}^\lambda$ is reversible, it corresponds to a random walk on $V\cup\{\partial\}$ with weights on the (undirected) edges given by
        \begin{align}\label{eq:defwlambda}
        w^\lambda(x,\partial) = \frac{\lambda}{1+\lambda} \cdot \pi(x) \quad \text{ and } \quad w^\lambda(x,y) = \frac{1-\lambda}{1+\lambda} \cdot P(x,y) \pi(x), \quad \forall \ x,y \in V
                \end{align}
        and we also write
        \[
        w^\lambda(x) = \sum_{y}w^\lambda(x,y) = \frac{\lambda}{1+\lambda}\cdot \pi(x)+ \frac{1-\lambda}{1+\lambda} \cdot \pi(x) = \frac{\pi(x)}{1+\lambda} \quad \text{ and } \quad w^\lambda(\partial) =\frac{\lambda}{1+\lambda}.
        \]
        So the total conductance of $V\cup \{\partial\}$ corresponding to $P^\lambda$ is equal to $\sum_x w^\lambda(x) =1$ and the stationary distribution of $P^{\la}$ is $w^{\la}$.
        
        We now consider the weighted graph $(V\cup \{\partial\}, (w^\lambda(x,y))_{(x,y)})$ and write $\reff^\lambda(x,y)$ for the effective resistance between $x$ and $y$. 
        Let $\eta^\lambda$ be a GFF on $V\cup\{\partial\}$ with $\eta^\lambda_\partial=0$ and 
        \[
        \E{(\eta_x^\lambda - \eta_y^\lambda)^2} = \reff^\lambda(x,y) \quad \text{ for all } \ x,y\in V\cup\{\partial\}.
        \]
        Let $X$ be a discrete time Markov chain on $V\cup\{\partial\}$ with matrix $P^\lambda$. We define for all $x\in V\cup\{\partial\}$ the hitting time
        \[
        \tau_x^\lambda =\min\{t\geq 0: X_t=x\}
        \]
        and the cover time 
        \[
        \taucov^\lambda = \max_x \tau_x^\lambda.
        \]
        We finally write $\tcov^\lambda = \max_x \estart{\taucov^\lambda}{x}$.

        \begin{lemma}\label{cl:monotonicity} 
Let $P$ be a finite, irreducible and reversible transition matrix. If $0<\lambda<\mu \le1/2$, then 
        \[
        \E{\max_v \eta_v^\mu}\leq 3\E{\max_v \eta_v^\lambda}.
        \] 
Moreover, $\tcov^\mu \lesssim \tcov^\lambda$.
\end{lemma}     
        
        \begin{proof}[\bf Proof]
By the Sudakov-Fernique inequality, Proposition~\ref{pro:sudakovfernique}, by taking the ground state to be $\partial$ in both networks it suffices to prove that for all $u$ and $v$
        \begin{align*}    
        \E{(\eta_v^\mu - \eta_u^\mu)^2} \leq 3 \E{(\eta_v^\lambda - \eta_u^\lambda)^2},
        \end{align*}
        which is equivalent to proving that for all $u$ and $v$
        \begin{align}\label{eq:comparingvar}
                \reff^\mu(u,v) \leq 3\reff^\lambda(u,v),
        \end{align}
        where $\reff^\lambda(x,y)$ stands for the effective resistance between $(x,y)$ in the network $(V\cup\{\partial\}, (w^\lambda(x,y))_{(x,y)})$. 
        Now notice that by the definition of the weights $w^\lambda$ in~\eqref{eq:defwlambda} it follows that for all $x$
        \[
        w^\lambda(x,\partial) < w^\mu(x,\partial)
        \]
        and for all $x,y\neq \partial$
        \[
        w^\lambda(x,y) \leq 3 w^\mu(x,y),
        \]
        since for $\mu\in (0,1/2]$ it holds that $(1-\mu)/(1+\mu) \ge 1/3$. Therefore using Thomson's principle we obtain for all $u,v$
        \[
        R^\lambda_{\rm{eff}}(u,v) \geq \frac{1}{3} R^\mu_{\rm{eff}}(u,v),
        \]
        which proves~\eqref{eq:comparingvar} and concludes the proof of the first statement of the claim. 
        
        For the second statement, by Theorem~\ref{thm:dingleeperes} (and $\sum_x w^{\la}(x)=1$) we get
        \begin{align}\label{eq:tcovlambdagff}
        \tcov^\lambda\asymp \left(\E{\max_x \eta_x^\lambda} \right)^2.
        \end{align}
        This combined with the first assertion of the lemma completes the proof.
    \end{proof}

        \begin{claim}\label{cl:lambdaequal}
        There exists a universal constant $C$ so that the following holds. Let $P$ be a finite irreducible and reversible transition matrix. Suppose that $\lambda$ satisfies $1/\lambda \geq 100\tcov$. Then 
        \[
        \tcov^\lambda \leq  \frac{C}{\lambda},
        \]
         where $\tcov$ corresponds to the Markov chain with matrix $P$.
\end{claim}

\begin{proof}[\bf Proof]

 Note that by definition, the chain with matrix ${P}^\lambda$ can be realised by letting a chain move on $V$ according to $P$ and then jump to $\partial$ after an independent geometric time of parameter~$\lambda$. After visiting $\partial$ it jumps to a state according to $\pi$ and then continues in $V$ according to $P$ until an independent geometric time again when it jumps to $\partial$ and so on.

We now consider excursions of this chain from $\partial$ and we will prove that with positive probability the walk covers the graph during one such excursion. Let $\Gamma$ be a geometric random variable of parameter $\lambda$ independent of $\taucov$. 
Recalling that $\taucov$ stands for the cover time of a chain on $V$ with transition matrix $P$ and using that $1/\lambda \geq 100\tcov$ we have 
\begin{align*}
        \prstart{\taucov^\lambda \leq 10\tcov}{\partial} \geq \prstart{\Gamma>10\tcov, \taucov < 10\tcov}{\pi}&\geq \pr{\Gamma>10\tcov} - \prstart{\taucov\geq 10\tcov}{\pi} \\ 
        &\geq 1- \frac{1}{10}- \frac{1}{10}=\frac{4}{5},
        \end{align*}
where for the third inequality we used Markov's inequality. Since  $\estart{\tau_\partial^\lambda}{x} = 1/\lambda$ for every $x\in V$, by the Markov property and the above it follows that there exists a positive constant $C$ so that 
\[
\tcov^\lambda\leq C\left( \tcov + \frac{1}{\lambda}\right)
\]
and this concludes the proof by the assumption on $\lambda$.
\end{proof}

We end this section with the following lemma which gives tail bound estimates on the local time of the auxiliary Markov chain. The proof is very similar to the proofs of Theorems~2.5 and~2.6 of~\cite{DingLeePeres}.

\begin{lemma}\label{lem:localtimes}
	There exist two positive constants $c$ and $C$ so that 
	for all $\delta\in (0,1/8)$, all~$\beta$ and~$A$ satisfying $A\wedge \beta\geq C$
	the following holds. Let $P$ be an irreducible and reversible with respect to $\pi$ transition matrix on the finite set $V$. 
	Let $\lambda\in (0,1)$ and $\eta^\lambda$ be a GFF on $(V\cup\{\partial\}, (w^\lambda(x,y))_{(x,y)})$ with boundary $\partial$.
	Let $t\geq \beta \left(\E{\max_x \eta^\lambda_x}\right)^2$ and
	let $Z$ have transition matrix~$P^\lambda$ and start from~$\partial$. For each $\ell \in \N$ and $v\in V$ let $N_\ell(v)$ be the number of visits to $v$ at the first time that $Z$ has completed $\ell+1$ visits to~$\partial$. Then we have
	\[
	\pr{\max_{x\in V} \frac{N_{\lfloor \lambda t\rfloor}(x)}{w^\lambda(x)} \geq A t} \lesssim e^{-c \beta} + e^{-c\lfloor\lambda t\rfloor}\quad \text{ and } \quad \pr{\min_{x\in V} \frac{N_{\lceil \lambda t\rceil}(x)}{w^\lambda(x)} \leq \delta t} \lesssim e^{-c \delta\beta} + e^{-c\lceil\lambda t\rceil}.
	\]
\end{lemma}

\begin{proof}[\bf Proof]
	
	In order to use the isomorphism theorem we need to pass to continuous time. We do so, by letting $\til Z$ be a continuous time Markov chain on $V\cup\{\partial\}$ which stays at each vertex for an independent exponential random variable of parameter $1$ and at the $i$-th jump time it goes to $Z_i$. We let $L_t(v)$ be the local time at $v$ by time $t$ for $\til{Z}$, i.e.\
	\[
	L_t(v) = \int_0^t \1(\til{Z}_s =v)\,ds.
	\]
	Let $\til{\tau}(t)$ be defined as
	\[
	\til{\tau}(t) = \inf\{ s\geq 0: L_s(\partial) \geq \lambda t\}.
	\]
	Note that since $w^\lambda(\partial)=\lambda/(1+\lambda)$ we have $\til{\tau}(t) = \tau((1+\lambda)t)$ with $\tau(t)$ as defined in Theorem~\ref{thm:isomorphism} corresponding to $\til{Z}$, i.e.
	\[
	\tau(t) = \inf\left\{s\geq 0: \frac{L_s(\partial)}{w^\lambda(\partial)}\geq t \right\}.
	\] 
	For each $i\geq 1$ and $x\in V\cup\{\partial\}$ we let $E_i^x$ be the time spent at $x$ on the $i$-th visit to $x$ by $\til{Z}$. Then $(E_i^x)_{i\geq 1, x}$ are i.i.d.\ exponential random variables of parameter $1$. We write $\til{N}_t(x)$ for the total number of times $\til{Z}$ visits $x$ by time $t$, i.e.
	\[
	\til{N}_t(x) = \min\left\{ i\geq 0:  L_t(x)\leq   \sum_{j=1}^{i}E_j^x  \right\}.
	\]
	For the remainder of the proof whenever we write $\max_x$ or $\min_x$ we mean $\max_{x\in V}$ and $\min_{x\in V}$.
	Using large deviations for the sum of independent exponential random variables we now have for a positive constant $c_1$
	\begin{align*}
		\pr{\max_x \frac{N_{\lfloor \lambda t\rfloor}(x)}{w^\lambda(x)}\geq At} \leq \pr{\sum_{i=1}^{\lfloor \lambda t\rfloor} E_i^\partial >  2\lfloor \lambda t\rfloor} + \pr{\max_x \frac{\til{N}_{\til\tau(2 t)}(x)}{w^\lambda(x)} \geq At} \\ \leq e^{-c_1\lfloor \lambda t\rfloor} +  \pr{\max_x \frac{\til{N}_{\til\tau(2 t)}(x)}{w^\lambda(x)} \geq At}.
	\end{align*}
	Passing from the number of visits to the continuous time local time we obtain
	\begin{align}\label{eq:maxvisits}
		\begin{split}
			\pr{\max_x \frac{\til{N}_{\til\tau(2t)}(x)}{w^\lambda(x)}\geq At}& \leq  \pr{\max_x \frac{L_{\til\tau(2t)}(x)}{w^\lambda(x)}\geq \frac{At}{2}} 
			\\
			&+ \pr{\max_x \frac{L_{\til\tau(2t)}(x)}{w^\lambda(x)}\leq \frac{At}{2}, \max_x \frac{\til{N}_{\til\tau(2t)}(x)}{w^\lambda(x)}\geq At}.
		\end{split}
	\end{align}
	For the second probability we have
	\begin{align*}
		\pr{\max_x \frac{L_{\til\tau(2t)}(x)}{w^\lambda(x)}\leq \frac{At}{2}, \max_x \frac{\til{N}_{\til\tau(2t)}(x)}{w^\lambda(x)}\geq At} 
		&\leq \sum_x \pr{\frac{L_{\til\tau(2t)}(x)}{w^\lambda(x)}\leq \frac{At}{2}, \frac{\til{N}_{\til\tau(2t)}(x)}{w^\lambda(x)}\geq At} \\ &\leq \sum_x \pr{\sum_{i=1}^{\lceil Atw^\lambda(x)\rceil} E_i^x \leq \frac{Atw^\lambda(x)}{2} }.
	\end{align*}
	Using large deviations for sums of exponentials gives that for a positive constant $c_2$
	\[
	\sum_x\pr{\sum_{i=1}^{\lceil At w^\lambda(x)\rceil} E_i^x \leq \frac{Atw^\lambda(x)}{2}}\leq \sum_x\exp\left(-c_2A tw^\lambda(x) \right) \lesssim e^{-c_3 A\beta},
	\]
	where $c_3$ is a positive constant and the second inequality follows from Lemma~\ref{lem:dingleeperes} by taking $A\beta$ sufficiently large. For the first probability appearing on the right hand side of~\eqref{eq:maxvisits} using Theorem~\ref{thm:isomorphism} we get 
	\begin{equation*}
		\begin{split}
			\pr{\max_x \frac{L_{\til\tau(2t)}(x)}{w^\lambda(x)} \geq \frac{At}{2}} & \leq \pr{\max_v \frac{1}{2}(\eta_v^\lambda +\sqrt{4(1+\lambda)t})^2\geq \frac{At}{2}}
			\\ & = \pr{\max_v |\eta_v^\lambda + \sqrt{4(1+\lambda)t}| \geq \sqrt{At}} \\ &
			\leq 2\pr{\max_v \eta^\lambda_v \geq \sqrt{t}(\sqrt{A} - \sqrt{4(1+\lambda)})}
			\\ & \leq 2\pr{\max_v \eta^\lambda_v \geq \sqrt{t}(\sqrt{A} - 2\sqrt{2})},
		\end{split}
	\end{equation*}
	where the last inequality follows since $\lambda\in (0,1)$. By taking $A$ and $\beta$ sufficiently large so that $\sqrt{A}-{2}\sqrt{2} - \frac{1}{\sqrt{\beta}}\geq c_4$, where $c_4$ is a positive constant, and
	using Borell's inequality for Gaussian processes we get 
	\begin{equation}\label{eq:borelineq}
		\begin{split}
			\pr{\max_v \eta^\lambda_v \geq \sqrt{t}(\sqrt{A} - 2\sqrt{2})} & \leq  2\exp\left(-\frac{t}{2\sigma^2} \left((\sqrt{A}-{2}\sqrt{2}) - \frac{1}{\sqrt{\beta}} \right)^2  \right)
			\\ & \leq 2\exp\left(-\frac{c_4^2t}{2\sigma^2}  \right),
		\end{split}
	\end{equation}
	where $\sigma^2=\max_x \E{(\eta^\lambda_x)^2}$. Note that $\sigma\leq \sqrt{2\pi} \E{\max_x \eta_x^\lambda}$. To see this, we repeat the argument given in~\cite{DingLeePeres}. Let $v$ be such that $\sigma^2=\E{(\eta_v^\lambda)^2}$. Then 
	\[
	\E{\max_x \eta_x^\lambda} \geq \E{\max(\eta_\partial^\lambda, \eta_v^\lambda)} = \E{\max(0,\eta_v^\lambda)} = \frac{\sigma}{\sqrt{2\pi}}.
	\]
	We finally get that the exponential appearing on the right hand side of~\eqref{eq:borelineq} can be bounded by
	\begin{align*}
		\pr{\max_v \eta^\lambda_v \geq \sqrt{t}(\sqrt{A} - {2}\sqrt{2})}\leq 2 e^{-c_5\beta},
	\end{align*}
	where $c_5$ is another positive constant. 
	This proves the bound on the first tail probability of the statement. 
	
	For the bound involving the minimum we have 
	\begin{equation*}
		\begin{split}
			\pr{\min_x \frac{N_{\lceil \lambda t\rceil}(x)}{w^\lambda(x)} \leq \delta t} & \leq \pr{\sum_{i=1}^{\lceil \lambda t\rceil}E_i^\partial <\frac{1}{2}\lceil \lambda t\rceil} + \pr{\min_x \frac{\til{N}_{\tau(t/2)}(x)}{w^\lambda(x)}\leq \delta t} \\ &
			\leq e^{-c_5 \lceil \lambda t\rceil} + \pr{\min_x \frac{\til{N}_{\tau(t/2)}(x)}{w^\lambda(x)}\leq \delta t}.
		\end{split}
	\end{equation*}
	Similarly to what we did above we get 
	\begin{align}\label{eq:minbounding}
		\begin{split}
			\pr{\min_x \frac{\til{N}_{\tau(t/2)}(x)}{w^\lambda(x)}\leq \delta t} &\leq \pr{\min_x \frac{L_{\til{\tau}(t/2)}(x)}{w^\lambda(x)}\leq 2\delta t} \\ & + \pr{\min_x \frac{L_{\til{\tau}(t/2)}(x)}{w^\lambda(x)}> 2\delta t, \min_x \frac{\til{N}_{\til{\tau}(t/2)}(x)}{w^\lambda(x)} \leq \delta t}.
		\end{split}
	\end{align}
	For the second term we obtain
	\begin{equation}\label{eq:boundonmin}
		\begin{split}
			\pr{\min_x \frac{L_{\til{\tau}(t/2)}(x)}{w^\lambda(x)}> 2\delta t, \min_x \frac{\til{N}_{\til{\tau}(t/2)}(x)}{w^\lambda(x)} \leq \delta t} & \leq \sum_x \pr{\sum_{i=1}^{\lfloor \delta t w^\lambda(x)\rfloor}E_i^x >2\delta t w^\lambda(x)} \\ & \lesssim e^{-c_6\delta \beta}
		\end{split}
	\end{equation}
	for a positive constant $c_6$ by taking $\beta$ sufficiently large, where for the last inequality we used again Lemma~\ref{lem:dingleeperes}.
	Returning to the first term appearing on the right hand side of~\eqref{eq:minbounding}, using Theorem~\ref{thm:isomorphism} again we obtain 
	\begin{align*}
		\begin{split}
			&	\pr{\min_x\frac{L_{\til\tau(t/2)}(x)}{w^\lambda(x)}\leq 2\delta t} \\ & \leq \pr{\max_x (\eta^\lambda_x)^2\geq \frac{(1-4\delta)}{2}t} + \pr{\min_x (\eta_x^\lambda+\sqrt{(1+\lambda)t})^2\leq \frac{(1+4\delta)}{2}t} 
			\\ &
			\leq 2\pr{\max_x \eta_x^\lambda \geq \sqrt{\frac{(1-4\delta)t}{2}}} + \pr{\min_x \eta_x^\lambda \leq -\sqrt{t}\left(\sqrt{1+\lambda}-\sqrt{\tfrac12+2\delta}\right)} 
			\\ &		\leq 2\exp\left(-\frac{t}{2\sigma^2}\left(\sqrt{\frac{1-4\delta}{2}}-\frac{1}{\sqrt{\beta}} \right)^2 \right) + 2\exp\left( - \frac{t}{2\sigma^2}\left(\sqrt{1+\lambda}-\sqrt{\frac{1}{2} + 2\delta} - \frac{1}{\sqrt{\beta}}  \right)^2 \right).
		\end{split}
	\end{align*}
	Using again that $\sigma\leq \sqrt{2\pi}\E{\max_x\eta_x^\lambda}$ we see that taking $\beta$ sufficiently large in terms of $\delta$ gives the upper bound of the statement. 
	Combining all the conditions on $\beta$ and $A$ we see that they can be replaced by the condition that $A\wedge \beta$ is sufficiently large and this concludes the proof.
\end{proof}

\section{Multiple walks}\label{sec:proofformultiplewalks}

In this section we start by proving Theorem~\ref{thm:indepone} by relating the cover time of $k$ particles to the cover time of the auxiliary Markov chain defined in Section~\ref{sec:modifiedchain} for an appropriate choice of $\lambda$. Then we prove Corollary~\ref{cor:worststartingstate} and Proposition~\ref{lem:loweboundontcovk}.

\begin{proof}[\bf Proof of Theorem~\ref{thm:indepone}]
	In this proof we write $\pr{\cdot}$ for the probability measure $\prstart{\cdot}{\pi^{\otimes k}}$ in order to lighten notation.
	We claim that it suffices to prove that there exist positive constants $c,C$ such that for all $k\leq \tcov$ we have 
	\begin{align}\label{eq:goalforallk}
		\pr{\taucov(k) \leq \left\lceil \frac{C\tcov}{k}\right\rceil}\geq c>0.
	\end{align}
	Indeed, once this is proved then for every $1\leq \ell\leq k$ we have 
	\begin{align*}
		\pr{\taucov(k) \geq \ell \left\lceil \frac{C\tcov}{k}\right\rceil} \leq \pr{\taucov(\lceil k/\ell\rceil) \geq \left\lceil \frac{C\tcov}{\lceil k/\ell\rceil}\right\rceil}^{\lfloor \ell/2\rfloor\wedge 1} \leq (1-c)^{\lfloor \ell/2\rfloor\wedge 1},
	\end{align*}
	where the second inequality follows from~\eqref{eq:goalforallk} and the first one by dividing the $k$ independent random walks into $\lfloor \ell/2\rfloor\wedge 1$ groups of at least $\lceil k/\ell\rceil$ walks each and using that these groups evolve independently. For $\ell>k$ we use that if $\taucov(k)>t$, then none of the $k$ chains has covered the state space by time $t$, and hence we have
	\begin{align*}
		\begin{split}
			\pr{\taucov(k) \geq \ell \left\lceil \frac{C\tcov}{k}\right\rceil} & \leq \pr{\taucov(k) \geq \ell  \frac{C\tcov}{k}} \leq       \pr{\taucov(1) \geq \left\lfloor\frac{\ell}{k}\right\rfloor C\tcov}^k
			\\ & \leq \left(\frac{1}{C} \right)^{\lfloor \ell/k\rfloor \cdot k},
		\end{split}
	\end{align*}
	where for the final inequality we used Markov's inequality and the Markov property. By taking the sum over all $\ell$ of the above two inequalities we get the desired bound for $\E{\taucov(k)}$ for $k\leq \tcov$.
	
	Now for $k>\tcov$ using that $\taucov(k)\leq \taucov(\tcov)$ and that from the above $\E{\taucov(\tcov)}\leq C'$ for a positive constant~$C'$ gives the desired result.
	
	Therefore it remains to prove~\eqref{eq:goalforallk} for all $k\leq \tcov$. It is enough to prove that there exists a positive constant $k_0$ sufficiently large so that for all $k_0\leq k\leq \tcov$ the bound~\eqref{eq:goalforallk} holds, since for $k<k_0$ we can obtain the desired bound by bounding the cover time of the $k$ chains by the cover time of a single chain.

	Let $\lambda = k/(C\tcov)$ for a positive constant $C>100$ to be determined. 
	
	Let $\ell_1,\ldots, \ell_k$ be i.i.d.\ excursions from $\partial$ for the Markov chain with state space $V\cup \{\partial\}$ and transition matrix $P^\lambda$. We write $\ell_i = (\partial, x_{1}^i,\ldots, x_{|\ell_i|}^i, \partial)$, where  $|\ell_i|+2$ stands for the length of $\ell_i$. For each $i$ we let $X^i$ be a Markov chain on $V$  that evolves as follows
	\[
	X^i_ j = x_{j}^i \quad \text{ for all } 1\leq j\leq |\ell_i| \quad \text{ and } X_{|\ell_i|+j}^i = Y^i_{j} \quad \text{ for } j\geq 0,
	\]
	where for every $i\leq k$ the process $(Y^i_j)_{j\geq 0}$ is an independent Markov chain on $V$ with transition matrix $P$ starting from $x_{|\ell_i|}^i$ and $(Y^i)_{i\leq k}$ are independent. From the definition of $P^\lambda$ we see that $x_{1}^i$ is distributed according to $\pi$ for all $i$, and hence we get that $(X^i)_{i\leq k}$ are i.i.d.\ and $X^i$ is distributed as a Markov chain on $V$ with transition matrix $P$ starting from $\pi$ for each $i$. Let $\taucov(k)$ be the first time that $V$ is covered by the union of the walks $(X^i)_{i\leq k}$, i.e.
	\[
	\taucov(k) = \inf\left\{m\geq 1: \cup_{i\leq k}\{X^i_j: j\leq m\} = V\right\}.
	\]
	For each $i$ we define $\til{\ell}_i$ as follows: 
	\begin{align*}
		\til{\ell_i} = 
		\begin{cases}
			\ell_i \quad &\text{ if } |\ell_i|<100C/\lambda \\
			(\partial, x_1^i,\ldots, x_{\lfloor 100C/\lambda\rfloor}^i) \quad &\text{ otherwise}.
		\end{cases}
	\end{align*}
	Recalling that $\lambda = k/(C\tcov)$ it is immediate that 
	\begin{align}\label{eq:inclusionfortaucov}
		\left\{ \taucov(k) \leq \frac{100C^2\tcov}{k}\right \} \supseteq \left\{ \bigcup_{i=1}^{k} \til{\ell}_i \supseteq V\right\}.
	\end{align}
	Let $L$ be the number of excursions $\ell_i$ with $|\ell_i|>100C/\lambda$, i.e.
	\[
	L = \sum_{i=1}^{k} \1(\til{\ell}_i \neq \ell_i).
	\]
	For $x\in V$, we write $V_x$ for the number of visits to $x$ by $\cup_{i\leq k} \ell_i$, where we identify $\ell_i$ with the set of points that it visits, and $V_x'$ for the number of visits to $x$ by $\cup_{i\leq k} \ell_i\setminus \til{\ell}_i$. Then $V_x-V_x'$ is the number of visits to $x$ by $\cup_{i\leq k}\til{\ell}_i$. Let $A$ be a positive constant to be fixed later. We then have that 
	\begin{align}\label{eq:vvx}
		\begin{split}
			\left\{ \bigcup_{i=1}^{k} \til{\ell}_i \supseteq V\right\} & = \left\{ \min_x (V_x-V_x')>0   \right\} \\
			&\supseteq \left\{ \min_{x} \frac{V_x}{w^\lambda(x)} \geq 2A\tcov    \right\}\cap \left\{ \max_{x} \frac{V_x'}{w^\lambda(x)}\leq A\tcov    \right\}.
		\end{split}
	\end{align}
	Let $Z$ start in $V$ according to $\pi$ and then evolve according to $P^\lambda$ independently of $L$. We write $N_\ell(v)$ for the number of visits to $v$ when $Z$ has completed $\ell$ visits to $\partial$. 
	By the memoryless property of the geometric distribution and the stationarity of $\pi$, the set $\cup_{i\leq k} \ell_i\setminus \til{\ell}_i$ has the same distribution as the set of vertices visited by $L$ independent excursions of $Z$ from $\partial$. So we get that $(V'_x)_x$ has the same distribution as $(N_L(x))_x$ and $(V_x)_x$ has the same distribution as $(N_k(x))_x$. We then have
	\begin{align*}
		\begin{split}
			\pr{\max_x \frac{V_x'}{w^\lambda(x)}\geq A\tcov } & = \pr{\max_x \frac{N_L(x)}{w^\lambda(x)}\geq A\tcov}
			\\ & \leq \pr{L> \left\lfloor \lambda \sqrt{A}\tcov\right\rfloor} + \pr{\max_x \frac{N_{\lfloor \lambda \sqrt{A}\tcov\rfloor}(x)}{w^\lambda(x)}\geq A\tcov}
			\\ &	\leq e^{-c_1 \sqrt{A}k/C } + \pr{\max_x \frac{N_{\lfloor \lambda \sqrt{A}\tcov\rfloor}(x)}{w^\lambda(x)}\geq A\tcov},
		\end{split}
	\end{align*}
	where $c_1$ is a positive constant and the first entry in the final inequality follows since $L$ is binomially distributed with parameters $k$ and $\pr{G(\lambda)>100C/\lambda}$ for $G(\lambda)$ a geometric random variable of parameter~$\lambda$. Using that $\lambda = k/(C\tcov)$ with $C>100$, we get by Lemma~\ref{cl:monotonicity} and Claim~\ref{cl:lambdaequal} that $\tcov\gtrsim \tcov^{\mu} \gtrsim \tcov^\lambda$ with $\mu=1/(100 \tcov)$. Also, by Theorem~\ref{thm:dingleeperes} we have 
	\[
	\tcov^\lambda \asymp \left(\E{\max_x \eta_x^\lambda} \right)^2.
	\]
	Therefore, applying Lemma~\ref{lem:localtimes} for $t=\sqrt{A}\tcov\gtrsim\sqrt{A} \tcov^\lambda$ we get that there exists a positive constant~$c$ so that for $A$ sufficiently large 
	(note that $\sqrt{A}$ plays the role of a multiple of $\beta$ in Lemma~\ref{lem:localtimes}) we have 
	\begin{align*}
		\pr{\max_x \frac{N_{\lfloor\lambda \sqrt{A}
					\tcov\rfloor}(x)}{w^\lambda(x)}\geq A\tcov}\lesssim e^{-c\sqrt{A}}+e^{-c\sqrt{A}k/C}.
	\end{align*}
	Putting everything together we get 
	\begin{align}\label{eq:maxboundfortaucov}
		\pr{\max_x \frac{V_x'}{w^\lambda(x)}\geq A\tcov } \lesssim e^{-c_1 \sqrt{A}k/C} + e^{-c\sqrt{A} }+ e^{-c\sqrt{A}k/C}.
	\end{align}
	Since $(V_x)_x$ has the same distribution as $(N_k(x))_x$, using Lemma~\ref{lem:localtimes} again there exists a positive constant $c'$ so that taking $C$ sufficiently large so that $2A/C<1/8$ we obtain 
	\begin{align}\label{eq:minvisits}
		\begin{split}
			\pr{\min_x \frac{V_x}{w^\lambda(x)} <2A\tcov}&= \pr{\min_x \frac{N_k(x)}{w^\lambda(x)} <2A\tcov} \\&=\pr{\min_x \frac{N_{\lambda C \tcov}(x)}{w^\lambda(x)} <2A\tcov} \lesssim e^{-c'C} + e^{-c'k/2}.
		\end{split}
	\end{align}
	Overall, using~\eqref{eq:inclusionfortaucov}, \eqref{eq:vvx}, \eqref{eq:maxboundfortaucov} and~\eqref{eq:minvisits} we deduce that there exists a positive constant $k_0$ and $A$ and $C$ sufficiently large with $C>16A$ such that for all $k\geq k_0$
	\begin{align}\label{eq:boundonprobtcov}
		\pr{\taucov(k)\leq \frac{100C\tcov}{k}} \geq \pr{\min_x \frac{V_x}{w^\lambda(x)}\geq 2A\tcov} - \pr{\max_x \frac{V_x'}{w^\lambda(x)} \geq A\tcov} \geq c_2>0,
	\end{align}
	where $c_2$ is a positive constant.
	This now concludes the proof of the theorem.
\end{proof}

We now state a standard result for reversible Markov chains that will be used in the next couple of proofs. 

\begin{lemma}{\rm{\cite[proof of Lemma~24.7]{LevPerWil}}}\label{lem:standard}
        Let $X$ be an irreducible and reversible Markov chain on a finite state space with transition matrix $(P+I)/2$ and invariant distribution $\pi$. For all $x$ there exists a probability measure $\nu_x$ so that for all $x$ and $y$ and $t=8\tmix$ we have 
        \[
        P^t(x,y) = \frac{3}{4}\pi(y) + \frac{1}{4} \nu_x(y).
        \]
\end{lemma}

\begin{proof}[\bf Proof of Corollary~\ref{cor:worststartingstate}]

Let $X^1,\ldots,X^k$ be $k$ independent Markov chains on $V$ with transition matrix $P$ starting from $x_1,\ldots, x_k$ respectively. Let $t=8\tmix$ and set $Y_i=X_t^i$ for all $i\leq k$. 
Applying Lemma~\ref{lem:standard} we get that for all $i\leq k$ we can sample $Y_i$ by first sampling a Bernoulli random variable $B_i$ with parameter $3/4$ and if it is equal to~$1$, then we sample $Y_i$ according to $\pi$. Otherwise, we sample it according to $\nu_{x_i}$. Let $N=\sum_{i=1}^k B_i$. Then $N$ has the binomial distribution with parameters $k$ and $3/4$. So we have 
\begin{align*}
         \estart{\tau_{\rm{cov}}(k)}{x_1,\ldots,x_k} &\leq C \tmix + \estart{\estart{\tau_{\rm{cov}}(k)}{Y_1,\ldots, Y_k}}{x_1,\ldots, x_k} \\ &\leq C \tmix + \estart{\estart{\tau_{\rm{cov}}(k)}{Y_1,\ldots, Y_k} \1(N>k/2)}{x_1,\ldots, x_k} + \tcov e^{-c k},
\end{align*}
where $c$ is a positive constant and for the last term we upper bounded the cover time of $k$ particles by the cover time $\tcov$ of a single particle and we also used large deviations for the binomial random variable $N$. To finish the proof we note that on the event $N>k/2$, there are at least $k/2$ variables among the $Y_i$'s that are distributed according to $\pi$. Since covering by $k$ independent chains is faster than covering by any subset of them, we obtain
\begin{align*}
        \estart{\estart{\tau_{\rm{cov}}(k)}{Y_1,\ldots, Y_k} \1(N>k/2)}{x_1,\ldots, x_k} \leq \estart{\tau_{\rm{cov}}(k/2)}{\pi^{\otimes k/2}} \lesssim \frac{\tcov}{k},
\end{align*}
where the last inequality follows from Theorem~\ref{thm:indepone} and this completes the proof.
\end{proof}

\begin{proof}[\bf Proof of Proposition~\ref{lem:loweboundontcovk}]

By Corollary~\ref{cor:worststartingstate} we have for all $k$
\[
\estart{\taucov(k)}{\pi^{\otimes k}} \leq \max_{x_1,\ldots,x_k}\estart{\taucov(k)}{x_1,\ldots,x_k}
\lesssim \frac{\tcov}{k}.
\]
To prove the lower bound, let $X$ be a Markov chain with transition matrix $P$ started from $\pi$ and let $\taucov$ be its cover time. Let~$t>0$ and set $t_i=it+ (i-1) 8\tmix$ for $i\geq 1$ and $Y_i=X_{t_i}$. By Lemma~\ref{lem:standard} again, for every $i$, conditional on $(X_j, j\leq t_{i-1}+t)$ we can sample $Y_i$ by first sampling a Bernoulli random variable of parameter $3/4$ and if it is equal to $1$, then taking $Y_i\sim \pi$, otherwise sampling~$Y_i$ according to the measure~$\nu_{X_{t_{i-1}}}$. Let $N$ be the number of $Y_i$'s for $i\leq k$ that are distributed according to $\pi$. Then $N$ can be stochastically dominated from below by a binomial random variable with parameters $k$ and~$3/4$. 
Using that 
\begin{align*}
        \{\taucov >kt + 8(k-2)\tmix \} \subseteq \bigcap_{1\leq i\leq k} \{ \{X_{t_{i-1}}, \ldots, X_{t_{i-1}+t}\} \neq V\}
\end{align*}
we obtain 
\begin{align*}
        \pr{\taucov>k t +8(k-2)  \tmix} &\leq \pr{\tau_{\rm{cov}}(k/2)>t} + \pr{N<k/2, \{X_{t_{k-1}}, \ldots, X_{t_{k-1}+t}\} \neq V} \\  &\leq \pr{\tau_{\rm{cov}}(k/2)>t} +
\pr{N<k/2} \max_x\prstart{\taucov>t}{x} \\&\leq \pr{\tau_{\rm{cov}}(k/2)>t} +
e^{-ck} \max_x\prstart{\taucov>t}{x} ,
\end{align*}
where $c$ is a positive constant and for the third inequality we used large deviations for $N$. Taking the sum over all $t$ we obtain
\begin{align}\label{eq:ineqrear}
        \frac{\tcov  - 8 (k-2) \tmix}{k} -1 \leq 
\estart{\taucov(k/2)}{\pi^{\otimes {k/2}}}       + e^{-ck} \sum_t\max_x \prstart{\taucov>t}{x}.
\end{align}
By the Markov property and Markov's inequality we get for all $\ell\in \N$
\[
\max_x\prstart{\taucov>2\ell\tcov}{x}\leq \frac{1}{2^\ell}.
\]
Therefore, we deduce that for a positive constant $C_1$
\[
\sum_t\max_x \prstart{\taucov>t}{x}\leq C_1\tcov.
\]
Using this, taking $C_2\leq k\leq \tcov/(16\tmix)$ for a positive constant $C_2$ and rearranging~\eqref{eq:ineqrear} yields the desired bound. To complete the proof for $k\leq C_2$ we use monotonicity. 
\end{proof}

\section{Proof of Theorem~\ref{thm: cover time lower bound}}
\label{sec:proofoftrel}

In order to prove Theorem~\ref{thm: cover time lower bound} we consider another family of auxiliary chains. These chains behave similarly to the ones we used in the proof of Theorem~\ref{thm:indepone}. Let $P$ be an irreducible and reversible transition matrix on the finite state space $V$ with invariant distribution $\pi$ and relaxation time~$\trel$. For the new auxiliary chains we do not add a new state to $V$ as in Section~\ref{sec:modifiedchain}, but instead at geometric times the chain jumps to a state chosen according to $\pi$. More formally, as in~\eqref{eq:defkp} for $\lambda\in [0,1]$ we define the matrix $K_\lambda$ via
\[
K_{\lambda}(x,y)=(1-\lambda)P(x,y)+\lambda \pi(y), \ \forall \ x,y. 
\]
It is immediate to check that for all $\lambda\in [0,1]$ the matrix $K_\lambda$ is reversible with respect to $\pi$. Therefore, $K_\lambda$ corresponds to a random walk on $V$ with weights on the undirected edges given by
\[
w_\lambda(x,y ) = (1-\lambda) \pi(x) P(x,y) + \lambda\pi(x)\pi(y), \ \forall \ x,y. 
\]
We write $w_\lambda(x)$ for the conductance of $x$. This is given by
\begin{align}\label{eq:wlambdapi}
w_\lambda(x) = \sum_y w_\lambda(x,y) = \pi(x), \ \forall \ x.
\end{align}
We write $\mathcal{R}_{\lambda}(x,y)$ for the effective resistance between $x$ and $y$ in the weighted graph $(V, w_\lambda)$. We write $\prstart{\cdot}{x,\lambda}$ and $\estart{\cdot}{x,\lambda}$ for the probability and expectation when the Markov chain starts from state $x$ and it has transition matrix $K_\lambda$. We write $\tcov(K_\lambda) = \max_x \estart{\taucov}{x,\lambda}$. Recall the notation $\tcov^\lambda$ for the expected cover time of the Markov chain on $V\cup \{\partial\}$ with transition matrix $P^\lambda$ as defined in Section~\ref{sec:modifiedchain}.

We first prove that the chains with matrices $P^\lambda$ and $K_\lambda$ for $\lambda\geq 1/\tcov$ have the same cover time up to constants. 

\begin{lemma}\label{lem:samecovertime}
        Let $\lambda\geq 1/\tcov$. Then 
        \[
        \tcov^\lambda \asymp \tcov(K_\lambda).
        \]
\end{lemma}
 
\begin{proof}[\bf Proof]

Let $x\in V$. We first describe a coupling between two walks on $V$ and $V\cup\{\partial\}$ started from $x$ and with transition matrices $K_\lambda$ and $P^\lambda$ respectively. We use the same geometric random variable of parameter $\lambda$ for both walks and we let them evolve together until the geometric time. At this time, the $K_\lambda$ chain jumps to a vertex $y$ with probability $\pi(y)$, while the chain $P^\lambda$ jumps to $\partial$ and immediately after we let it jump to the vertex the chain with matrix $K_\lambda$ already jumped to. Then we use the same jumps for $K_\lambda$ as for $P^\lambda$ until the next geometric time where we do the same as before. This shows that if $\taucov(P,A)$ stands for the first time that a chain with transition matrix~$P$ covers the set $A$, then for all $x$ we get 
\begin{align}\label{eq:vlambdaplambda}
\estart{\taucov(K_\lambda,V)}{x} \leq \estart{\taucov(P^\lambda, V)}{x}\leq 2\estart{\taucov(K_\lambda,V)}{x}.
\end{align}
Since starting from any $x$ the chain with matrix $P^\lambda$ hits $\partial$ after a geometric time of parameter $\lambda$, we obtain
\begin{align}\label{eq:boundwithzandwithout}
\estart{\taucov(P^\lambda, V)}{x} \leq \estart{\taucov(P^\lambda,V\cup \{\partial\})}{x} \leq  \estart{\taucov(P^\lambda, V)}{x} + \frac{1}{\lambda}.
\end{align}
We next show that there exists $x\in V$ such that 
\begin{align*}
        \frac{1}{\lambda}\lesssim \estart{\taucov(P^\lambda,V)}{x}.
\end{align*}
Recall that the walk with matrix $P^\lambda$ can be realised by running a random walk on $G$ for a geometric time of parameter $\lambda$ at which time it jumps to $\partial$. At the next step it jumps to a state according to $\pi$ and then continues in the same way. 
Writing $\taucov$ for the cover time of a random walk on $G$, taking $x$ such that $\tcov = \estart{\taucov}{x}$ and letting $\Gamma$ be an independent  geometric random variable of parameter $\lambda$ we get for a constant~$C>2$ to be chosen and using that $1/\lambda\leq \tcov$ \begin{align}\label{eq:boundwithP}
        \prstart{\taucov(P^\lambda,V) \geq \frac{1}{C\lambda}}{x} \geq \prstart{\Gamma\geq \frac{1}{C\lambda}, \taucov>\frac{\tcov}{2}}{x} \geq \prstart{\taucov>\frac{\tcov}{2}}{x} -\frac{1}{C}.
\end{align}
Using the Markov property and Markov's inequality one immediately obtains that for all $\ell\in \N$
\[
\max_y\prstart{\taucov >2\ell\tcov}{y} \leq \frac{1}{2^\ell}.
\]
This now implies that $\max_y\estart{\taucov^2}{y}\leq C_1 (\tcov)^2$ for a positive constant $C_1$, and hence using the Payley Zygmund inequality we deduce
\[
\prstart{\taucov>\frac{\tcov}{2}}{x} \geq  \frac{1}{4C_1}.
\]
Plugging this lower bound into~\eqref{eq:boundwithP} and by choosing $C$ sufficiently large gives 
\begin{align*}
        \prstart{\taucov(P^\lambda,V) \geq \frac{1}{C\lambda}}{x} \geq c,
        \end{align*}
where $c$ is a positive constant. 
Therefore, this proves that 
\[
\estart{\taucov(P^\lambda, V)}{x}\gtrsim \frac{1}{\lambda},
\]
which together with~\eqref{eq:boundwithzandwithout} implies that 
\[
\max_{y\in V}\estart{\taucov(P^\lambda, V)}{y} \asymp \max_{y\in V}\estart{\taucov(P^\lambda, V\cup \{\partial\})}{y}.
\]
Moreover, we have
\[
\estart{\taucov(P^\lambda, V)}{\partial} =  1 + \estart{\taucov(P^\lambda, V)}{\pi},
\]
and hence using also~\eqref{eq:vlambdaplambda} we deduce
\[
\max_{y\in V\cup \{\partial\}} \estart{\taucov(P^\lambda, V\cup \{\partial\})}{y} \asymp \max_y \estart{\taucov(K_\lambda, V)}{y}.
\]
This finishes the proof.
\end{proof}

Next we state a theorem comparing the effective resistances of $K_0= P$ to the one of $K_\lambda$ when $\lambda\leq 1/\trel$. This result is the main ingredient in the proof of Theorem~\ref{thm: cover time lower bound} but it is also of independent interest. 

\begin{theorem}
\label{thm:effective resistance comparison}
There exists an absolute constant $c>0$ (independent of $P$) such that in the above setup and notation we have that for all $0 \le \la \le 1/\rel$ and all $x, y \in V$ with $x\neq y$
\begin{equation}
\label{e:comparisonofresistancewithrel}
c\mathcal{R}_{0}(x,y) \le \mathcal{R}_{\lambda}(x,y) \le \frac{1}{1-\lambda}\cdot \mathcal{R}_{0}(x,y).
\end{equation}
\end{theorem}

We now give the proof of Theorem~\ref{thm: cover time lower bound} deferring the proof of Theorem~\ref{thm:effective resistance comparison} to Section~\ref{sec:comparison}.

\begin{proof}[\bf Proof of Theorem~\ref{thm: cover time lower bound}]
Theorem~\ref{thm:effective resistance comparison}, the Sudakov-Fernique inequality~\ref{pro:sudakovfernique}  and Theorem~\ref{thm:dingleeperes} give that for $\lambda\leq 1/\trel$
\begin{align}\label{eq:firstpandklambda}
\tcov \lesssim \tcov(K_\lambda).
\end{align}
From~\eqref{eq:firstpandklambda} and~Lemma~\ref{lem:samecovertime} we get that there exist two positive constants $c_1$ and $c_2$ so that for all $1/\tcov\leq \lambda\leq 1/\trel$
\begin{align}\label{eq:tcovplambda}
c_2\tcov\leq \tcov^\lambda\leq c_1 \tcov.
\end{align}
In order to finish the proof it suffices to show that there exists a positive constant $C\geq 5$ so that if $
 C\leq k\leq c_2\tcov/(10\trel)$ and $\lambda=10k/(c_2\tcov)$ we have that 
\begin{align}\label{eq:goallowertcovk}
\tcov(k) \gtrsim \frac{\tcov^\lambda}{k},
\end{align}
where the constants in $\gtrsim$ are independent of $\lambda$ and $k$. Indeed, for such $\lambda$ and $k$ we can use~\eqref{eq:tcovplambda} and get the desired bound. For $k<C$ we use monotonicity of $\tcov(k)$ to finish the proof, i.e.\ for $k<C$
\[
\E{\taucov(k)}\geq \E{\taucov(C)}\gtrsim \frac{\tcov}{C}\asymp \frac{\tcov}{k}.
\]
So we now prove~\eqref{eq:goallowertcovk}.
Let $\ell_1, \ell_2, \ldots$ be i.i.d.\ excursions from $\partial$ for the chain with matrix $P^\lambda$. Then the lengths of the $\ell_i$'s are i.i.d.\ each of them distributed as $2+\rm{Geo}(\lambda)$ with ${\rm{Geo}}(\lambda)$ a geometric random variable of parameter $\lambda$. Concatenating these excursions gives us a realisation of the chain with matrix $P^\lambda$ started from $\partial$.  Let $C_1$ be a positive constant to be determined.  Consider the first $k$ excursions with length larger than $\lceil 1/(C_1\lambda)\rceil+2$. Their first $\lceil 1/(C_1\lambda)\rceil$ steps (not including the starting vertex $\partial$) give a realisation of $k$ independent walks on~$G$ started from $\pi$ run for $\lceil 1/(C_1\lambda)\rceil$ steps. So we have 
\begin{align*}
        \prstart{\taucov(k)\geq \left\lceil\frac{1}{C_1\lambda}\right\rceil}{} \geq \prstart{\sum_{i=1}^{2k}\1\left(|\ell_i|\geq \left\lceil\frac{1}{C_1\lambda}\right\rceil+2\right)\geq k, \sum_{i=1}^{2k}|\ell_i|\leq \frac{4k}{\lambda}, \taucov^\lambda>\frac{4k}{\lambda}}{\partial}.
\end{align*}
By the definition of the chain $P^\lambda$ we have 
\[
\tcov^\lambda \leq \frac{1}{\lambda} + \estart{\taucov^\lambda}{\partial}
\]
and since $\lambda = 10k/(c_2\tcov)$ and $k\geq 5$ we get 
\[
\estart{\taucov^\lambda}{\partial}
\geq \tcov^\lambda - \frac{1}{\lambda} \geq \tcov^\lambda\left( 1 - \frac{1}{k}   \right) \geq 
 \frac{4}{5} \cdot \tcov^\lambda.
\]
Hence applying the Paley Zygmund inequality again as in the proof of Lemma~\ref{lem:samecovertime} we get that there exists a positive constant $c$ so that
\[
\prstart{\taucov^\lambda>\frac{4k}{\lambda}}{\partial} \geq \prstart{\taucov^\lambda>\frac{2}{5}\tcov^\lambda}{\partial} \geq c>0.
\]
By large deviations for the binomial and the sum of geometric random variables, we now obtain that there exists a positive constant $c'$ so that taking $C_1$ sufficiently large
\[
\pr{\sum_{i=1}^{2k}\1\left(|\ell_i|\geq \left\lceil\frac{1}{C_1\lambda}\right\rceil+2\right)\geq k, \sum_{i=1}^{2k}|\ell_i|\leq \frac{4k}{\lambda}} \geq 1-2e^{-c'k}.
\]
Therefore, there exists a positive constant $C$ so that for $k\geq C$ we get that 
\[
\prstart{\sum_{i=1}^{2k}\1\left(|\ell_i|\geq \left\lceil\frac{1}{C_1\lambda}\right\rceil+2\right)\geq k, \sum_{i=1}^{2k}|\ell_i|\leq \frac{4k}{\lambda}, \taucov^\lambda>\frac{4k}{\lambda}}{\partial}\geq c''
\]
for a positive constant $c''$. This shows that
\[
\E{\taucov(k)} \gtrsim \frac{\tcov}{k}
\]
for $k\geq C$.  This  concludes the proof.   
\end{proof}

\subsection{Comparison of effective resistances}\label{sec:comparison}

This section is devoted to the proof of Theorem~\ref{thm:effective resistance comparison}. Let $X$ be a Markov chain with transition matrix $P$. For a state $x$ and a time $t\in \N$ we write $N(x,t)$ for the number of visits to $x$ up to time $t$, i.e.\
\[
N(x,t)= \sum_{i=0}^{t} \1(X_i=x).
\]
Recall that $\tau_a$ denotes the first hitting time of $a$ by $X$, i.e.\ $\tau_a = \min\{t\geq 0: X_t= a\}$. 
We start with a couple of preliminary standard results.

\begin{lemma}
\label{lem:auxexpdecay} Let $P$ be an irreducible and reversible transition matrix on the finite set $V$ with invariant distribution $\pi$ and relaxation time $\trel$. Let $Q=(P+I)/2$ be the lazy version of $P$.  Then for every state $x$ and all $M > 0$ we have that
\begin{equation}
\label{e:auxexp1}
\sum_{k=0}^{\infty}(Q^{k}(x,x)-\pi(x)) \le \frac{e^{M/2}}{e^{M/2}-1}    \sum_{k=0}^{\lceil M\rel\rceil}(Q^{k}(x,x)-\pi(x)).
\end{equation}
\end{lemma}

\begin{proof}[\bf Proof]
First of all we note that the relaxation time of $\trel(Q)$ of the matrix $Q$ satisfies
\[
\trel(Q) = 2\trel.
\]
It follows from the spectral decomposition (e.g., \cite[\S12.1]{LevPerWil}) that for all $x$ and  all $s,t \ge 0$ we have that   
 \begin{equation}
 \label{e:specd}
0< Q^{t+s}(x,x) -\pi(x) \le e^{-s/\rel(Q)}(Q^{t}(x,x) -\pi(x)) = e^{-s/(2\rel)}(Q^{t}(x,x) -\pi(x)) . 
\end{equation} 
Hence defining 
\[
f(i)=\sum_{k=i\lceil M \rel\rceil }^{(i+1)\lceil  M\rel\rceil}(Q^k(x,x)-\pi(x)) 
\]
we get that  $\frac{ f(i)}{ f(0) } \le e^{-Mi/2} $ for all $i \in \mathbb{N}$. This now immediately implies the statement of the lemma. 
\end{proof}

\begin{claim}\label{cl:standardclaim}
        Let $P$ be an irreducible and reversible transition matrix on the finite state space $V$ with invariant distribution $\pi$. Then for $s\geq 4\estart{\tau_a}{\pi}$ we have
        \[
        \tv{P^s_L(a,\cdot) -\pi}\leq \frac{1}{4},
        \]
        where $P_L=(P+I)/2$.
\end{claim}

A stronger inequality is proven in \cite[Eq.\ (1.2)]{hermon2019some}.

\begin{proof}[\bf Proof]
We write $\til{\tau}_a$ for the first hitting time of $a$ by the chain with matrix $P_L$. In the proof of~\cite[Theorem~10.22]{LevPerWil} it is shown that 
\[
\frac{\estart{\til{\tau}_a}{\pi}}{t} \geq \left|\frac{P_L^t(a,a)}{\pi(a)}-1 \right|.
\]
This together with the inequality
\[
\tv{P_L^t(a,\cdot)-\pi}^2\leq \frac{1}{4} \|P_L^t(a,\cdot)-\pi\|_{2,\pi}^2 = \frac{1}{4} \left(\frac{P_L^{2t}(a,a)}{\pi(a)} -1 \right)
\]      
concludes the proof, since $\estart{\til{\tau}_a}{\pi}=2\estart{\tau_a}{\pi}$.
\end{proof}

\begin{lemma}
\label{returnstoy}
Let $P$ be an irreducible and reversible transition matrix on the finite state space $V$ with invariant distribution $\pi$ and relaxation time $\trel$. Let $a\neq b \in V$ and assume that $\mathbb{P}_{\pi}(\tau_a<\tau_{b}) \le 1/2$. Then
\begin{equation}
\label{e:N1}
\mathbb{E}_{a}[N(a,\tau_b)] \le 4\estart{N(a,\lfloor 8\estart{\tau_a}{\pi}\rfloor)}{a}\leq 
\frac{24e^{1/8}}{e^{1/8}-1} \mathbb{E}_{a}[N(a,\lceil\rel/4\rceil)]. 
\end{equation}
\end{lemma}

\begin{proof}[\bf Proof]

It suffices to prove both inequalities for the lazy version of $P$, i.e.\ the Markov chain with transition matrix $P_L=(P+I)/2$, but where $\estart{\tau_a}{\pi}$ and $\trel$ still refer to the chain with matrix $P$. Indeed, the quantities of~\eqref{e:N1} corresponding to $P_L$ would differ from the ones corresponding to~$P$ by a factor of $2$. We write $\til{\tau}_a$ for the hitting time of $a$ by the chain with matrix $P_L$.

It is standard that for any two distributions $\mu, \nu$ we have that
\[|\mathbb{P}_{\mu}(\til{\tau}_b<\til{\tau}_{a}) -\mathbb{P}_{\nu}(\til{\tau}_b<\til{\tau}_{a})| \le \|\nu-\mu\|_{\mathrm{TV}}.\]
Set $s=\lfloor 8\estart{\tau_a}{\pi}\rfloor$ and $\nu(\cdot)=P^{s}_L(a,\cdot)$. Then by the assumption that $\prstart{\tau_a<\tau_b}{\pi}\leq 1/2$, it follows that $s\geq 4\estart{\tau_a}{\pi}$, and hence we can apply Claim~\ref{cl:standardclaim} to obtain
\[
\|\nu - \pi\|_{\rm{TV}} \leq \frac{1}{4}.
\]
Therefore, this implies
 \[
 \mathbb{P}_{\nu}(\til{\tau}_b<\til{\tau}_{a}) \ge\mathbb{P}_{\pi}(\til{\tau}_b<\til{\tau}_{a})  -1/4=\prstart{\tau_b<\tau_a}{\pi}-\frac{1}{4}> 1/4. 
 \]
We have the obvious bound
\begin{align}\label{eq:obviousbound}
        \estart{N(a,\tau_b)}{a} \leq \estart{N(a,s)}{a} + \estart{N(a,\tau_b)}{\nu}.
\end{align}
By the strong Markov property
\[
 \estart{N(a,\tau_b)}{\nu}= \prstart{\tau_a<\tau_b}{\nu} \estart{N(a,\tau_b)}{a} \leq \frac{3}{4} \cdot  \estart{N(a,\tau_b)}{a}.
\]
Substituting this bound into~\eqref{eq:obviousbound} gives the first inequality in \eqref{e:N1}. 

For the second inequality, we start by writing 
\begin{align*}
        \estart{N(a,s)}{a} = \sum_{k=0}^{s} P_L^k(a,a)= \sum_{k=0}^s (P_L^k(a,a)-\pi(a) ) + (s+1) \pi(a).
\end{align*}
Since the chain is lazy, it follows that $P_L^t(x,x)\geq \pi(x)$ for all $x$ and $t$ (see for instance~\cite[Proposition~10.25]{LevPerWil}). Therefore, by Lemma~\ref{lem:auxexpdecay}
\begin{align*}
        \sum_{k=0}^s (P_L^k(a,a)-\pi(a) ) \leq \sum_{k=0}^\infty (P_L^k(a,a)-\pi(a) ) \leq \frac{e^{1/8}}{e^{1/8}-1}\cdot \sum_{k=0}^{\lceil\trel/4\rceil} P_L^k(a,a).
\end{align*}
Writing $\til{\tau}_a$ for the first hitting time of $a$ for the chain with matrix $P_L$, we have that 
\[
\estart{\til{\tau}_a}{\pi} = 2\estart{\tau_a}{\pi}\geq s/4.
\]
By~\cite[Proposition~10.26]{LevPerWil} we have that for all $x$
\[
\pi(x) \estart{\til\tau_x}{\pi} = \sum_{k=0}^\infty (P_L^k(x,x) - \pi(x)).
\]
Hence we conclude that 
\begin{align*}
        \estart{N(a,s)}{a} \leq \pi(a)+ 4\pi(a)\estart{\til{\tau}_a}{\pi} +\frac{e^{1/8}}{e^{1/8}-1} \sum_{k=0}^{\lceil\trel/4\rceil}P_L^k(a,a) \\ 
        \leq  \frac{6e^{1/8}}{e^{1/8}-1} \sum_{k=0}^{\lceil\trel/4\rceil}P_L^k(a,a)
        = \frac{6e^{1/8}}{e^{1/8}-1} \cdot \estart{N(a,\lceil\trel/4\rceil)}{a}.
\end{align*}
This finishes the proof of the second inequality. 
\end{proof}

\begin{lemma}
\label{lem:someeasyobserations}
Let $P$ be an irreducible transition matrix on the finite set $V$ and $\la\in [0,1]$. Let $\Gamma_{\lambda}$ be a geometric random variable of parameter~$\lambda$ independent of the chain. Let $a\neq b \in V$. 
\begin{itemize}
\item[\rm(i)] If $\mathbb{P}_{a}(\tau_b < \Gamma_{\lambda}  ) \ge 1/2$, then $\mathbb{E}_a[N(a,\tau_b)] \le 2 \mathbb{E}_a[N(a,\Gamma_{\lambda}\wedge \tau_b)]$.
\item[\rm(ii)] If  $\mathbb{P}_{a}(\tau_b < \Gamma_{\lambda} ) \mathbb{P}_{b}(\tau_a < \Gamma_{\lambda} )\le 1/4$,  then $\mathbb{E}_a[N(a,\Gamma_{\la})] \le \frac{4}{3} \mathbb{E}_a[N(a, \Gamma_{\la} \wedge \tau_b)]$. 
\end{itemize}
\end{lemma}
\begin{proof}[\bf Proof]
For part (i) we let $\nu(\cdot )=\prcond{X_{\Gamma_{\lambda}}= \cdot}{\Gamma_{\lambda}  <\tau_b}{a}$. Then
\begin{align*}
\mathbb{E}_a[N(a,\tau_b)] -\mathbb{E}_a[N(a, \Gamma_{\la} \wedge \tau_b)] = \mathbb{E}_a[ \left( N(a,\tau_{b})-N(a, \Gamma_{\la} \wedge \tau_b)\right)\1(\Gamma_{\lambda} <\tau_b ) ]\\
 = \prstart{  \Gamma_{\lambda}  <\tau_b}{a}\mathbb{E}_{\nu}[N(a,\tau_b)]\le \frac{1}{2}\mathbb{E}_{\nu}[N(a,\tau_b)] = \frac{1}{2}\mathbb{E}_{a}[N(a,\tau_b)]\prstart{\tau_{a}  <\tau_b}{\nu}.  
 \end{align*}
Rearranging yields part (i).
We now prove part (ii).
By the strong Markov property and the memoryless property of the geometric distribution, we have that
\begin{align*}
\mathbb{E}_a[N(a, \Gamma_{\la})]-\mathbb{E}_a[N(a, \Gamma_{\la} \wedge \tau_b)]=\mathbb{E}_a\left[ \left( N(a,\Gamma_{\la})-N(a, \Gamma_{\la} \wedge \tau_b)\right)\1(\Gamma_{\lambda} >\tau_b) \right]  
\\
=\mathbb{P}_{a}(\tau_b < \Gamma_{\lambda}) \mathbb{P}_{b}(\tau_a < \Gamma_{\lambda} ) \mathbb{E}_a[N(a,\Gamma_{\la})] \le \frac 14 \mathbb{E}_a[N(a,\Gamma_{\la})].   
\end{align*}
Rearranging yields part (ii) and finishes the proof.
\end{proof}

\begin{proof}[\bf Proof of Theorem~\ref{thm:effective resistance comparison}]
        
        We start by proving the easy direction, i.e.\ that for all $x,y$
        \begin{align}\label{eq:easydirectioneff}
   \forall \lambda \in (0,1], \quad     \mathcal{R}_\lambda(x,y) \leq \frac{1}{1-\lambda} \cdot \mathcal{R}_0(x,y).
        \end{align}
        By the definition of the weights $w_\lambda$ we immediately get for all edges $(x,y)$ of the graph
        \[
        w_\lambda(x,y) \geq (1-\lambda)w_0(x,y).
        \]
        Therefore, using Thomson's principle for effective resistances immediately yields~\eqref{eq:easydirectioneff}.
        
We now prove the more interesting part of theorem. Namely, that there exists an absolute constant $c>0$ such that $c\mathcal{R}_{0}(x,y) \le \mathcal{R}_{\lambda}(x,y)$ whenever $0 \le \la \le 1/\rel$. Recall that the effective resistance satisfies for all $a,b$
\begin{align}\label{eq:defeffres}
\cR_\lambda(a,b) = \frac{1}{w_\lambda(a)\prstart{\tau_b<\tau_a^+}{a,\lambda}} = \frac{\estart{N(a,\tau_b)}{a,\lambda}}{w_\lambda(a)}.
\end{align}
Let $0<\la  \le 1/\rel$.  Since the effective resistance is symmetric in its arguments (i.e.\ $\cR_{\la}(x,y)=\cR_{\la}(y,x)$ for all $\la \ge 0$) and since $w_\lambda(a) = \pi(a)$ for all $\lambda$ (see~\eqref{eq:wlambdapi}), it suffices to show that
\[\mathbb{E}_{x,0}[N(x,\tau_y)] \le \frac{44e^{1/8}}{e^{1/8}-1} \mathbb{E}_{x,\la}[N(x,\tau_y)] \quad \text{or} \quad \mathbb{E}_{y,0}[N(y,\tau_x)] \le \frac{44e^{1/8}}{e^{1/8}-1} \mathbb{E}_{y,\la}[N(y,\tau_x)].   \]
We assume without loss of generality that $\mathbb{P}_{\pi}(\tau_x<\tau_y) \le 1/2$. Let $\Gamma_{\la}$ be as in Lemma~\ref{lem:someeasyobserations}.  We may assume that  $\mathbb{P}_{x}(\tau_y <\Gamma_{\lambda}  ) < 1/2$ and  $\mathbb{P}_{y}(\tau_x < \Gamma_{\lambda}  ) < 1/2$, as otherwise by part (i) of Lemma~\ref{lem:someeasyobserations} there is nothing to prove, since either
\[
\estart{N(x,\tau_y)}{x,0} 
\leq  2\estart{N(x,\tau_y\wedge \Gamma_\lambda)}{x} \leq 2\estart{N(x,\tau_y)}{x,\lambda}
\]
or the same inequality with the roles of $x$ and $y$ reversed would hold.

 Hence we are in the setup of part (ii) of Lemma  \ref{lem:someeasyobserations}, which implies that $\mathbb{E}_{x,0}[N(x,\Gamma_{\lambda})] \le \frac{4}{3}\mathbb{E}_{x,\la}[N(x,\tau_y)] $.  Accordingly, it suffices to show that
\[\mathbb{E}_{x,0}[N(x,\tau_y)] \le \frac{32e^{1/8}}{e^{1/8}-1} \mathbb{E}_{x,0}[N(x,\Gamma_{\lambda})].    \]  
By the assumption
 $\mathbb{P}_{\pi}(\tau_x<\tau_y) \le 1/2$ and the fact that $\la \le 1/\rel$ together with Lemma~\ref{returnstoy}, we get that
\[\mathbb{E}_{x,0}[N(x,\tau_y)] \le \frac{24e^{1/8}}{e^{1/8}-1}\mathbb{E}_{x,0}[N(x,\lceil\rel/4\rceil)]  \le \frac{32e^{1/8}}{e^{1/8}-1} \mathbb{E}_{x,0}[N(x,\Gamma_{\lambda})]. \]
Substituting this bound into~\eqref{eq:defeffres} concludes the proof of the theorem.
\end{proof}

\bibliography{biblio}

\begin{thebibliography}{10}

\bibitem{AF}
D.~Aldous and J.~Fill.
\newblock {\em Reversible Markov Chains and Random Walks on Graphs}.
\newblock In preparation,
  http://www.stat.berkeley.edu/$\sim$aldous/RWG/book.html.

\bibitem{AlonKozmaetal}
N.~Alon, C.~Avin, M.~Kouck\'{y}, G.~Kozma, Z.~Lotker, and M.~R. Tuttle.
\newblock Many random walks are faster than one.
\newblock {\em Combin. Probab. Comput.}, 20(4):481--502, 2011.

\bibitem{frogs1}
I.~Benjamini, L.~R. Fontes, J.~Hermon, and F.~P. Machado.
\newblock On an epidemic model on finite graphs.
\newblock {\em Ann. Appl. Probab.}, 30(1):208--258, 2020.

\bibitem{Biskupnotes}
M.~Biskup.
\newblock Extrema of the two-dimensional discrete {G}aussian free field.
\newblock In {\em Random graphs, phase transitions, and the {G}aussian free
  field}, volume 304 of {\em Springer Proc. Math. Stat.}, pages 163--407.
  Springer, Cham, [2020] \copyright 2020.

\bibitem{BroderKarlinetal}
A.~Z. Broder, A.~R. Karlin, P.~Raghavan, and E.~Upfal.
\newblock Trading space for time in undirected {$s$}-{$t$} connectivity.
\newblock {\em SIAM J. Comput.}, 23(2):324--334, 1994.

\bibitem{Chatterjee}
S.~Chatterjee.
\newblock Spectral gap of nonreversible {M}arkov chains.
\newblock arXiv:2310.10876.

\bibitem{cooperfriezeetal}
C.~Cooper, A.~Frieze, and T.~Radzik.
\newblock Multiple random walks in random regular graphs.
\newblock {\em SIAM J. Discrete Math.}, 23(4):1738--1761, 2009/10.

\bibitem{DingLeePeres}
J.~Ding, J.~R. Lee, and Y.~Peres.
\newblock Cover times, blanket times, and majorizing measures.
\newblock {\em Ann. of Math. (2)}, 175(3):1409--1471, 2012.

\bibitem{EfremenkoReingold}
K.~Efremenko and O.~Reingold.
\newblock How well do random walks parallelize?
\newblock In {\em Approximation, randomization, and combinatorial
  optimization}, volume 5687 of {\em Lecture Notes in Comput. Sci.}, pages
  476--489. Springer, Berlin, 2009.

\bibitem{isomorphismthm}
N.~Eisenbaum, H.~Kaspi, M.~B. Marcus, J.~Rosen, and Z.~Shi.
\newblock A {R}ay-{K}night theorem for symmetric {M}arkov processes.
\newblock {\em Ann. Probab.}, 28(4):1781--1796, 2000.

\bibitem{Sauer}
R.~Els\"{a}sser and T.~Sauerwald.
\newblock Tight bounds for the cover time of multiple random walks.
\newblock {\em Theoret. Comput. Sci.}, 412(24):2623--2641, 2011.

\bibitem{Feige}
U.~Feige.
\newblock A tight lower bound on the cover time for random walks on graphs.
\newblock {\em Random Structures Algorithms}, 6(4):433--438, 1995.

\bibitem{hermon2019some}
J.~Hermon.
\newblock Some inequalities for reversible {M}arkov chains and branching random
  walks via spectral optimization.
\newblock {\em Ann. Inst. Henri Poincar\'e{} Probab. Stat.}, 58(3):1650--1668,
  2022.

\bibitem{Thomas1}
A.~Iva\v{s}kovi\'c, A.~Kosowski, D.~Pajak, and T.~Sauerwald.
\newblock Multiple random walks on paths and grids.
\newblock In {\em 34th {S}ymposium on {T}heoretical {A}spects of {C}omputer
  {S}cience}, volume~66 of {\em LIPIcs. Leibniz Int. Proc. Inform.}, pages Art.
  No. 44, 14. Schloss Dagstuhl. Leibniz-Zent. Inform., Wadern, 2017.

\bibitem{LevPerWil}
D.~A. Levin and Y.~Peres.
\newblock {\em Markov chains and mixing times}.
\newblock American Mathematical Society, Providence, RI, 2017.
\newblock Second edition of [MR2466937], With contributions by Elizabeth L.
  Wilmer, With a chapter on ``Coupling from the past'' by James G. Propp and
  David B. Wilson.

\bibitem{Thomas}
N.~Rivera, T.~Sauerwald, and J.~Sylvester.
\newblock Multiple random walks on graphs: mixing few to cover many.
\newblock In {\em 48th {I}nternational {C}olloquium on {A}utomata, {L}anguages,
  and {P}rogramming}, volume 198 of {\em LIPIcs. Leibniz Int. Proc. Inform.},
  pages Art. No. 107, 16. Schloss Dagstuhl. Leibniz-Zent. Inform., Wadern,
  2021.

\bibitem{Zhai}
A.~Zhai.
\newblock Exponential concentration of cover times.
\newblock {\em Electron. J. Probab.}, 23:Paper No. 32, 22, 2018.

\end{thebibliography}
\bibliographystyle{abbrv}

\end{document}